\documentclass[12pt, a4paper, reqno, twoside, makeidx]{amsart}
\usepackage[margin=2.8cm, marginpar=3cm]{geometry}
\usepackage{comment}
\usepackage{verbatim}
\usepackage{amsmath}
\usepackage{amsfonts}
\usepackage{amssymb}
\usepackage{graphicx}
\usepackage{color}
\usepackage{empheq}
\usepackage[font={small,it}]{caption}
\usepackage[all]{xy}
\usepackage{tikz-cd}
\usepackage{tikz}
\usetikzlibrary{calc}
\usepackage{amssymb}
\usepackage{amsmath, amsthm, amssymb, amsfonts, amscd}
\usepackage[english]{babel}
\usepackage[all]{xy}
\usepackage{url}
\usepackage{hyperref}
\usepackage{mathtools}

\theoremstyle{plain}

\newtheorem{thm}{Theorem}[section]
\newtheorem{lem}[thm]{Lemma}
\newtheorem{prop}[thm]{Proposition}
\newtheorem{defi}[thm]{Definition}
\newtheorem{cor}[thm]{Corollary}

\newtheorem{rem}[thm]{Remark}
\theoremstyle{definition}

\newcommand{\N}{\mathbb{N}}
\newcommand{\Z}{\mathbb{Z}}
\newcommand{\Q}{\mathbb{Q}}
\newcommand{\R}{\mathbb{R}}

\newcommand{\K}{\mathcal{K}}

\newcommand{\Laur}{\mathbb{Z}_2[U,U^{-1}]}

\newcommand{\F}{\mathbb{Z}_2}

\newcommand{\x}{\mathbf{x}}
\newcommand{\y}{\mathbf{y}}
\newcommand{\gr}{\text{gr}}

\newcommand{\s}{\mathfrak{s}}

\newcommand{\Spin}{\text{Spin}}
\newcommand{\Spinc}{\text{Spin}^c}

\newcommand*{\QEDB}{\hfill\ensuremath{\square}}

\newcommand{\CFK}{\mathcal{CFK}}
\newcommand{\XX}{\mathbb{X}}
\newcommand{\OO}{\mathbb{O}}

\newcommand{\modd}[1]{\;(\text{mod } #1)}
\newcommand{\spinc}{\mbox{spin}^c}

\newcommand{\lp}{L(p,q)}
\newcommand{\lu}{L(p,1)}
\newcommand{\CFKinf}{{\mathcal {CFK}}^{\infty}}
\newcommand{\CFKinfty}{\CFKinf}
\newcommand{\mathcalCFK}{{\mathcal {GCFK}}}

\newcommand{\twistknot}{TW}

\author{Antonio Alfieri} 
\address{University of British Columbia, Mathematics Department, 1984 Mathematics Rd, Vancouver, BC V6T 1Z2, Canada}
\email{alfieriantonio90@gmail.com}

\author{Daniele Celoria}
\address{Mathematical Institute, University of Oxford, Radcliffe Observatory Quarter, Woodstock Rd, Oxford OX2 6GG, UK}
\email{Daniele.Celoria@maths.ox.ac.uk}

\author{Andr\'as Stipsicz}
\address{R\'enyi Institute of Mathematics, 1053. Budapest, Re\'altanoda
utca 13-15. Hungary}
\email{stipsicz@renyi.hu}

\begin{document}
\title[]{Upsilon invariants from cyclic branched covers}
\begin{abstract} 
We extend the construction of $\Upsilon$-type invariants to null-homologous knots in rational homology three-spheres. By considering $m$-fold cyclic branched covers with $m$ a prime power, this extension provides
new knot concordance invariants $\Upsilon_m^C(K)$ of knots in $S^3$. We give computations of some of these  invariants for 
alternating knots and reprove independence results in the smooth concordance group.
\end{abstract}
\maketitle
\thispagestyle{empty}
\section{Introduction}
\label{sec:intro}

Knot Floer homology turned out to be an extremely powerful tool in
studying properties of knots. For $K\subset S^3$ the theory provides
invariants which determine the Seifert genus of the knot~\cite{OS1},
whether it is fibered or not~\cite{Ghi, YN} (both achieved by the
${\widehat {HFK}}$-theory), and further versions of the construction
give lower bounds on the slice genus (the $\tau$ invariant stemming
from the $HFK^{-}$-theory).  The most general version (the graded, 
$\Z\oplus \Z$-filtered 
$\F [U, U^{-1}]$-module $\CFKinf$) gave rise to additional
invariants, like Hom's $\epsilon$-invariant \cite{Hom}, or the
$\Upsilon$-function of \cite{OSS4}, or its further variants $\Upsilon
^C$ from \cite{alfieri1}.  More recently, based on work of Hendricks
and Manolescu \cite{manoleacuendriks} involutive invariants emerged as
useful tools in the theory. Most of these constructions worked for
knots in $S^3$ and provided strong results on the structure of the
smooth concordance group ${\mathcal {C}}$.

In \cite{OSS4} Ozsv\' ath, Szab\' o and the third author pointed out
the possibility of capturing more information about concordance of
knots in $S^3$ by running the routine of the upsilon invariant for
their preimages in branched covers. In this paper we pursuit that idea
and perform some computations of the resulting invariants. Notice that
this option was already explored by Grigsby, Ruberman, and Strle
\cite{grigsby2008knot} for the Ozsv\' ath-Szab\' o $\tau$ invariant.

In the first part of the paper we extend the theory of upsilon type
invariants to null-homologous knots in rational homology spheres
(compare with \cite{hom2018knot}) and we apply it to branched covers
to obtain invariants of knots in $S^3$. For a knot $K\subset S^3$, let
$\Sigma^m(K)$ denote the $m$-fold cyclic branched cover along $K$ for some
prime power $m$.  A knot $K\subset S^3$, a cohomology class $\xi \in
H^2(\Sigma^m(K); \Z)$, and a south-west region $C$ of the plane
provides a knot invariant $\Upsilon_\xi^C(K)$ with the following key
property.

\begin{thm}\label{theorem2} 
If $K$ is a slice knot then there exists a subgroup $G < H^2(\Sigma^m(K); \Z)$ of cardinality
$\sqrt{|H^2(\Sigma^m(K); \Z)|}$ such that $\Upsilon^C_\xi(K)=0$ for all
$\xi\in G$.  
\end{thm}

The second part of the paper is devoted to computations. We extend
Grigsby's result \cite[Thm. 4.3]{grigsby2006combinatorial} to
encompass all $\spinc$ structures, in the subcase of alternating torus
knots. (Using some general principles, similar results can be deduced
-- here we apply direct computational methods.)

\begin{thm}\label{thm:interi}
Let $p = 2n+1$ be a given positive odd integer. For $h \in \{0,\ldots,
n\}$ there exist bi-graded quasi-isomorphisms
$$CFK^\circ (L(p,1), \widetilde{T}_{2,p}, \s_0 \pm h) \simeq  CFK^\circ (T_{2,p-2h}),$$
where $\circ = \wedge, \pm, \infty$.
\end{thm}

This result determines $HFK^\circ(L(p,1), \widetilde{T}_{2,p})$ as a graded group,
and implies that the lift of an alternating torus knot is \emph{thin}
in each $\spinc$ structure (cf. Section $6$ of
\cite{levine2008computing}). 

Using \cite[Lemmas~5~and~7]{Petkova}, this result implies that the
corresponding graded, $\Z\oplus \Z$-filtered chain complexes $\CFKinf (L(p,1),
\widetilde{T}_{2,p}, \s_0 \pm h)$ and $\CFKinf (T_{2,p-2h})$ are
graded, $\Z \oplus \Z$-filtered chain homotopy equivalent. Thus, as a
consequence of the previous computation, we can easily obtain the
following results. (More general results, encompassing all alternating
knots, will be presented in~\cite{pacana}.)

\begin{cor}\label{cor:tautorici} 
For an alternating torus knot $K= T_{2,2n+1}$  we have that for
$t\in [0,2]$
\[ \ \ \ \  \ \ \ \  \ \ \ \ \ \  \ \Upsilon_{K, \s_0+ h}(t)= (|h|-n) \cdot (1-|t-1|) \ \ \ \ \ \  -n \leq h \leq n \ . \]
\end{cor}

Theorem \ref{thm:interi} can also be applied to compute the $\tau$-invariants studied in \cite{raoux2016tau}.
\begin{cor}
For $h=0, \dots , n$ one has
\[\tau_{\s_0+h}(\widetilde{T}_{2,2n+1}) = \tau(T_{2,2n+1-2 h})= - \frac{1}{2} \cdot \sigma(T_{2,2n+1-2 h}) = n - h,\]  
\end{cor}

We then work out partial calculations in the case of twist
knots.  Further computations are going to appear in
\cite{alfieri2}, using lattice cohomology techniques.

\bigskip

{\bf {Acknowledgements}}: DC acknowledges support from the
European Research Council (ERC) under the European Unions Horizon 2020
research and innovation programme (grant agreement No 674978). AA and SA
acknowledge support from the NKFIH Grant \emph{\'Elvonal}
(\emph{Frontier}) KKP 126683 and from K112735. The authors want to thank the anonymous referees for helpful comments and suggestions.

\section{An invariance principle}
\label{sec:invariance}

An \textbf{Alexander filtered, Maslov graded chain
  complex} is a finitely-generated, $\Q$-graded, $\Z \oplus
\Z$-filtered chain complex $\mathcal{K}_*= (\bigoplus_{\x \in B}
\F[U, U^{-1}], \partial )$ over the ring $\F[U,U^{-1}]$ satisfying the following properties
\begin{itemize}
\item $\partial$ is $\Laur$-linear, and given a basis element $\x \in
  B$, $\partial \x = \sum_\y n_{\x, \y}U^{m_{\x,\y}} \cdot \y$ for
  suitable coefficients $ n_{\x, \y} \in \F$, and non-negative
  exponents $m_{\x, \y} \geq 0$,
\item the multiplication by $U$ drops the homological (Maslov) grading $M$ by two, and the filtration levels (denoted by $A$ and $j$) by one.
\end{itemize} 
An Alexander filtered, Maslov graded chain complex is said to be of
\textbf{knot type} if in addition $H_*(\mathcal{K}_*, \partial)= \F[U,
  U^{-1}]$ graded so that $\text{gr}(1)=d(\K_*)$, for some
$d(\mathcal{K}_*)\in \Q$. The number $d(\mathcal{K}_*)$ is the
\textbf{correction term} of $\mathcal{K}_*$.  An Alexander filtered,
Maslov graded chain complex $\mathcal{K}_*$ can be pictorially
described as follows:
\begin{itemize}
\item represent each $\F$-generator $U^m \cdot \x$ of $\mathcal{K}_*$ as a point on the planar lattice $\Z \times \Z \subset \R^2$ in position 
$\left(-m, A(\x)-m \right) \in \Z \times \Z$,
\item label each $\Z_2$-generator $U^m \cdot \x$ of $\mathcal{K}_*$ with its Maslov grading $M(\x)-2m\in \Z$,
\item connect two $\F$-generators $U^n \cdot \x$ and $U^m \cdot \y $ with a directed arrow if in the differential of $U^n \cdot \x$ the coefficient of $U^m \cdot \y$ is non-zero.
\end{itemize}

Two knot type complexes $\mathcal{K}_1$ and $\mathcal{K}_2$ are said
to be \textbf{stably equivalent}, denoted by $\mathcal{K}_1 \sim
\mathcal{K}_2$, if there exist two graded, $\Z\oplus \Z$-filtered
acyclic chain complexes $\mathcal{A}_1$ and $\mathcal{A}_2$ such that
$\mathcal{K}_1 \oplus \mathcal{A}_1 \simeq \mathcal{K}_2 \oplus
\mathcal{A}_2$, where $\simeq$ denotes (graded, $\Z\oplus \Z$-filtered) 
chain homotopy
equivalence. Denote by $\CFK$ the set of knot type complexes up to
chain homotopy equivalence. The quotient set $\CFK /_\sim$
has a natural group structure, with operation given by tensor product,
and identity represented by the chain complex 
$\F [U, U^{-1}]$, equipped with $\partial =0$ and the trivial
filtration; the inverse of the class of a complex $\mathcal{K}_*$ is
represented by its dual complex $\text{Hom}(\mathcal{K}_*,\F[U,
  U^{-1}])$.  Let $\mathcalCFK$ denote this Abelian group. 

Let $K$ be a null-homologous oriented knot in a rational homology
sphere ($\Q HS^3$) $Y$.  Knot Floer homology \cite{OS7}
associates to the pair $(Y,K)$ a finitely generated, $\Q$-graded, $\Z
\oplus \Z$-filtered chain complex 
$$\CFKinf (Y, K)= \left(\bigoplus_{\x \in B} \F[U, U^{-1}] \cdot \x, \partial \right)$$ 
which is an Alexander filtered, Maslov graded chain complex in the above sense. 

For each basis element $\x \in B$, there is an associated $\spinc$
  structure $\s(\x) \in \Spinc (Y)$, and for each $\spinc$ structure
  $\s$ of $Y$
\[
\CFKinf (Y,K, \s) = \bigoplus_{\s(\x)= \s} \F[ U, U^{-1}] \cdot \x
\] 
is a subcomplex of the knot Floer complex $\CFKinf (Y, K)$,
and
\[
\CFKinf (Y,K) = \bigoplus_{\s\in \Spinc (Y)}  \CFKinf (Y,K, \s). 
\]
The chain
complex $\CFKinf (Y,K, \s) $ satisfies 
\[H_*(\CFKinf (Y, K, \s))= HF^\infty(Y, \s)=\Laur\] 
graded so that $\text{gr}(1)= d$, where
$d=d(Y, \s)$ denotes the correction term of $(Y,\s)$ as defined in
\cite{OS24}. In conclusion, 
$\CFKinf (Y,K, \s)$ is a chain complex of knot type.

In \cite{OS7} Ozsv\' ath and Szab\' o proved that the $\Q$-graded,
$\Z \oplus \Z$-filtered chain homotopy type of $\CFKinf (Y, K,
\s)$ only depends on the diffeomorphism type of the pair $(Y,K)$, and
on the chosen $\spinc$ structure $\s$.  By forgetting the $A$-filtration,
and taking the homology of the $\F[U]$-submodule $j\leq 0$ we get the
invariant $HF^-(Y,\s )$ of the ambient 3-manifold $Y$.
If $(Y,\s)$ is a $\spinc$ $\Q HS^3$, then
$HF^-(Y,\s) = \F [U]_{(d)}\oplus HF_{red}(Y,\s )$, where 
$\F [U]_{(d)}$ is graded
so that $1$ is in degree $d=d(Y,\s)$; this is commonly referred to as a \emph{tower}. 
The second summand $HF_{red}(Y,\s)$ is the \emph{reduced} Heegaard Floer
homology of $(Y,\s)$, and it is a 
finitely generated $\F$-module.

A coarser equivalence relation among the pairs $(Y,K)$, generalizing
usual knot concordance, is defined as follows (see also \cite{CHA} for
a general reference on rational concordance).
\begin{defi}\label{ratconcordance}
 For $i=0,1$ let $Y_i$ be a rational homology sphere, $K_i \subset
 Y_i$ a null-homologous knot, and $\s_i \in \Spinc (Y_i)$. We will say
 that $K_0$ and $K_1$ are \textbf{$\spinc$ rationally concordant}, if
 there exists a smooth $\spinc$ rational homology cobordism $(W,
 \mathfrak{t})$ from $(Y_0, \s_0)$ to $(Y_1, \s_1)$, and a smoothly
 properly embedded cylinder $C \subset W$ such that $\partial C= C
 \cap \partial W = K_1 \sqcup -K_0$.
\end{defi}

A null-homologous knot $K$ in a $\spinc$ rational homology sphere
$(Y,\s)$ is \textbf{rationally slice} if there exists a $\spinc$
rational homology ball $(W, \mathfrak{t})$ bounding $(Y, \s)$
containing a smoothly properly embedded disk $\Delta \subset W$ such
that $\partial \Delta = \Delta \cap \partial W=
K$. 

Let $\mathcal{C}_\Q$ denote the set of triples $(Y, K,
\s)$, where $Y$ is a rational homology sphere, $K \subset Y$ is a
null-homologous knot, and $\s \in \Spinc (Y)$, considered up to
rational concordance. $\mathcal{C}_\Q$ has a group structure induced
by connected sum
\[(Y_0, K_0, \s_0)\# (Y_1, K_1, \s_1)= (Y_0\# Y_1, K_0\# K_1, \s_0\# \s_1) \ . \]  
In this group structure rationally slice knots represent the trivial element, and the inverse corresponds to taking the mirror image $-(Y, K,
\s)= (-Y, -K, \overline{\s})$. Here $-Y$ stands for the three-manifold $Y$ with its reversed
orientation, and $-K$ is the knot $K$ with its orientation reversed.

The main goal of this section is to prove the following theorem,
generalizing 
a result of Hom~\cite[Theorem 1]{hom2017survey} to rational homology spheres,
\emph{cf}.~also \cite[Section~4]{hom2018knot}.
\begin{thm} \label{invariance} 
Let $(Y_0,K_0,\s_0)$ and $(Y_1,K_1,\s_1)$ be two $\spinc$ rationally
concordant knots. Then the chain complexes $\CFKinf (Y_0, K_0, \s_0)$
and $\CFKinf (Y_1, K_1, \s_1) $ are stably equivalent, that is, 
there exist $\Q$-graded, $\Z \oplus
\Z$-filtered, acyclic chain complexes ${\mathcal {A}}_0$ and
${\mathcal {A}}_1$ (that is, $H_*({\mathcal {A}}_i)=0$) such
that $$\CFKinf (Y_0, K_0, \s_0) \oplus {\mathcal {A}}_0 \simeq \CFKinf
(Y_1, K_1, \s_1) \oplus {\mathcal {A}}_1,$$ where $\simeq$ denotes
graded $\Z\oplus \Z$-filtered chain homotopy equivalence.
\end{thm} 

Theorem~\ref{invariance} can be deduced from the following lemma.

\begin{lem} \label{rationalknots} 
Let $K$ be a null-homologous knot in a rational homology sphere $Y$, and $\s \in \Spinc (Y)$. If $K$ is rationally slice, then there exists a $\Q$-graded, $\Z \oplus \Z$-filtered, acyclic chain complex ${\mathcal {A}}$ such that 
\[
\CFKinf (Y, K, \s)\simeq \CFKinfty (S^3, U, \mathfrak{u}) \oplus
        {\mathcal {A}} \ ,
\]
where $U \subset S^3$ denotes the unknot, and $\mathfrak{u}$ denotes
the unique $\spinc$ structure of $S^3$.
\end{lem}  
\begin{rem}
It is not hard to see that 
$\CFKinf (S^3, U, \mathfrak{u}) \simeq \F [U, U^{-1}]$ with $\partial =0$,
$gr(1)=0$ and both the Alexander and algebraic filtration level of $1$ is
zero.
\end{rem}

\begin{proof}[Proof of Theorem \ref{invariance}] 
For $i=0,1$ let $Y_i$ be a rational homology sphere, $K_i \subset Y_i$
a null-homologous knot, and $\s_i \in \Spinc (Y_i)$. Note that 
\begin{itemize}
\item if $(Y_0, K_0, \s_0)$ and $(Y_1, K_1, \s_1)$ represent
  rationally concordant null-homologous knots then $(Y_1 \# - Y_0,
  K_1\#- K_0, \s_1 \# \overline{\s}_0)$ is rationally slice,
\item according to \cite{OS7} for $(Y_0, K_0, \s_0)$, $(Y_1, K_1,
  \s_1)$ null-homologous knots we have $\CFKinf (Y_0 \# Y_1, K_0 \#
  K_1, \s_0\# \s_1)\simeq \CFKinf (Y_0, K_0, \s_0) \otimes_{\Laur}
  \CFKinf (Y_1, K_1, \s_1)$ and the gradings and filtrations add.
\end{itemize}

 Suppose that
$K_0$ and $K_1$ are rationally concordant, and consider the chain
complex
\[
\mathcal{K}=\CFKinf (Y_0 \# (-Y_1) \# Y_1, K_0 \# (-K_1) \#K_1, \s_0\#  \overline{\s}_1 \# \s_1).
\] 
As consequence of Lemma \ref{rationalknots}, we have that
\begin{align*}
\mathcal{K} &\simeq \CFKinf(Y_0, K_0, \s_0) \otimes
\CFKinfty((-Y_1)\#Y_1, (-K_1)\# K_1, \overline{\s}_1 \# \s_1)
\\ &\simeq \CFKinfty(Y_0, K_0, \s_0) \otimes \left( \CFKinfty (S^3, U,
\mathfrak{u} ) \oplus {\mathcal {A}} \right) \\ &\simeq \left( \CFKinfty (Y_0,
K_0, \s_0) \otimes \CFKinfty(S^3, U, \mathfrak{u}) \right) \oplus
\left( \CFKinfty(Y_0, K_0, \s_0) \otimes {\mathcal {A}} \right)\\ &\simeq
\CFKinfty (Y_0 \# S^3, K \# U, \s_0 \# \mathfrak{u}) \oplus
\left(\CFKinfty (Y_0, K_0, \s_0) \otimes {\mathcal {A}} \right) \\ &= \CFKinfty
(Y_0, K_0, \s_0) \oplus \left(\CFKinfty (Y_0, K_0, \s_0) \otimes {\mathcal {A}}
\right),
\end{align*} 
for some acyclic complex ${\mathcal {A}}$. On the other hand, 
\begin{align*}
\mathcal{K}
&\simeq \CFKinfty (Y_0\#(-Y_1), K_0\# (-K_1), \s_0 \# \overline{\s}_1) \otimes  
\CFKinfty (Y_1, K_1, \s_1)\\
&\simeq \left( \CFKinfty (S^3, U, \mathfrak{u} ) \oplus {\mathcal {B}} \right) \otimes  
\CFKinfty (Y_1, K_1, \s_1) \\
&\simeq  \left(  \CFKinfty (S^3, U, \mathfrak{u}) \otimes  \CFKinfty (Y_1, K_1, \s_1)   \right)  \oplus \left( {\mathcal {B}} \otimes \CFKinfty (Y_1, K_1, \s_1) \right) \\
&\simeq  \left(  \CFKinfty (S^3 \# Y_1, U \# K_1, \mathfrak{u} \# \s_1)   \right)  \oplus \left( {\mathcal {B}} \otimes \CFKinfty (Y_1, K_1, \s_1) \right) \\
&=  \CFKinfty (Y_1, K_1, \s_1) \oplus  \left( {\mathcal {B}} \otimes \CFKinfty (Y_1, K_1, \s_1) \right), 
\end{align*}
again for some acyclic complex ${\mathcal {B}}$. Thus 
\[  \CFKinfty (Y_0, K_0, \s_0) \oplus  {\mathcal {A}}_0 \simeq \CFKinfty (Y_1, K_1, \s_1) \oplus  {\mathcal {A}}_1 \ ,\]
where ${\mathcal {A}}_0= \CFKinfty (Y_0, K_0, \s_0) \otimes {\mathcal
  {A}}$ and ${\mathcal {A}}_1= {\mathcal {B}} \otimes \CFKinfty (Y_1, K_1,
\s_1)$. Using the K\" unneth formula one concludes that ${\mathcal
  {A}}_0$ and ${\mathcal {A}}_1$ are both acyclic.
\end{proof}

For the proof of Lemma~\ref{rationalknots} we need a little
preparation.  Let $K \subseteq Y$ and $\s$ be as above, and consider the knot Floer
complex $\CFKinfty (Y, K, \s)$. For $m \geq 0$ set
\begin{equation}\label{eqn:V}
V_K(m, \s)= d(Y, \s) -2 \cdot \min_{i} \max(A(z_i)-m, j(z_i)) \ ,
\end{equation}
where $z_1, \dots , z_k \in \CFKinfty (Y,K,\s)$ are the cycles with
Maslov grading $d=d(Y, \s)$ representing the non-zero element of
$H_{d}(\CFKinfty (Y,K,\s))= \F$. Our first goal is to relate $V_K(m,
\s)$ to the correction terms of the surgeries along $K$. 

Given a null-homologous knot $K\subset Y$, define its
\emph{four-dimensional genus} $g_*(K)$ as the minimal genus of a
smooth, compact surface in $Y \times [0,1]$ with boundary $K\times \{
0\}$.  For an integer $q$, let $W_q(K)$ denote the $q$-framed
two-handle attachment along $K \times \{1\} \subset Y \times [0,1]$,
so that $\partial W_q(K) = Y_q(K) \sqcup - Y$.  Furthermore, for any
integer $m \in [-q/2, q/2]$, let $\mathfrak{t}_m$ denote the $\spinc$
structure on $W_q(K)$ extending $\s$ to $Y_q(K)$, and satisfying $
\langle c_1(\mathfrak{t}_m), [\widehat{F} ] \rangle +q=2m$, where
$\widehat{F} \subset W_q(K) $ denotes a Seifert surface $F$ for $K$,
capped-off with the core of the $2$-handle.  Finally, let
$\mathfrak{s}_m\in \Spinc(Y_q(K))$ denote the $\spinc$ structure we
get by restricting $\mathfrak{t}_m$ to $Y_q(K)$. With these notations
and definitions in place, we have the following result.

\begin{prop}\label{correctionterms}
Let $K \subseteq Y$ be a null-homologous knot in a rational homology
sphere $Y$, $\s\in \Spinc (Y)$, and fix an integer $q\geq 2g_*(K)
-1$.   Then
\[ d(Y_q(K), \mathfrak{s}_m)= \frac{(q-2m)^2-q}{4q}+ V_K(m, \s) \ .\]
\end{prop} 
\begin{proof} Let $\CFKinfty (Y,K,\s)\{ A\leq m,j \leq 0\}$ denote
the subcomplex of $\CFKinfty (Y,K)$ spanned by generators with
Alexander filtration level $A\leq m$, and algebraic filtration level
$j \leq 0$. According to \cite[Section 4]{OS7} we have that:
\begin{itemize}
\item $\CFKinfty (Y,K,\s)\{ A \leq m, j \leq 0\}$ is chain homotopy
  equivalent to $CF^-(Y_q(K), \mathfrak{s}_m)$;
\item $\CFKinfty (Y,K, \s)\{ j \leq 0\}$, the subcomplex of
  $\CFKinfty(Y,K)$ spanned by the generators with algebraic
  filtration level $j \leq 0$, is chain homotopy equivalent to the
  Heegaard Floer complex $CF^-(Y, \s)$ of the ambient $\spinc$ $3$-manifold; 
\item the inclusion $\CFKinfty(Y,K,\s)\{ A \leq m, j\leq 0\}
  \hookrightarrow \CFKinfty(Y,K, \s)\{ j \leq 0\}$ descends in
  homology to the map $F_{X,\mathfrak{t}_m}\colon HF^-(Y_q(K),
  \mathfrak{s}_m) \to HF^-(Y, \s)$ induced by the cobordism
  $X=-W_q(K)$ endowed with the $\spinc$ structure $\mathfrak{t}_m$.
\end{itemize}
By considering the degree-shift formula of~\cite{ozsvath2003absolutely}, we obtain
\[\gr(F_{X,\mathfrak{t}_m}(\xi) )- \gr(\xi)=  \frac{c_1(\mathfrak{t}_m)^2- 2 \chi(X) - 3 \sigma(X)}{4}\ ,\]
where $\xi$ denotes the generator of the free summand of $HF^-(Y_q(K),
\mathfrak{s}_m)$. Thus
\[d(Y_q(K), \mathfrak{s}_m)= d+\frac{(q-2m)^2-q}{4q}   \ , \]
where $d$ is the grading of the generator of the free summand of
$H_*(\CFKinfty(Y,K,\s)\{ A \leq m, j \leq 0\})$.  Since the inclusion
\[
\CFKinfty(Y,K,\s)\{ A \leq m, j\leq 0\} \hookrightarrow CFK^\infty(Y,K, \s)
\{ j \leq 0\} 
\] 
maps the generator of the free summand of
$HF^-(Y_q(K), \mathfrak{s}_m)$ to a $U^n$-multiple of the one of
$HF^-(Y , \s)$, if $z_1, \dots , z_k \in \CFKinfty(Y, K, \s)$ denote
the cycles with Maslov grading $M=d(Y, \s)$ representing $1 \in \F[U,
  U^{-1}]= H_*(\CFKinfty(Y, K, \s))$ then 
\begin{equation*}
 d= \max_i M(U^{n_i}\cdot z_i)= \max_i (M(z_i)- 2n_i)= d(Y, \s) -2 \min_i n_i ,
\end{equation*}
   where $n_i$ is the
minimum $n\geq 0$ such that $U^n \cdot z_i \in \CFKinfty(Y,K,\s)\{ A
\leq m, j \leq 0\}$.  

Since $n_i=\max(A(z_i)-m,j(z_i))$ this proves
that 
\[
d=d(Y, \s) -2\min_i \max(A(z_i)-m ,j(z_i))=V_K(m, \s).
\]
\end{proof}

\begin{cor} Let $K \subset Y$ be a knot in a rational homology 
sphere, $\s \in Spin^c(Y)$ and $m \in \N$. Then $V_K(m, \s)$ is a
$\spinc$ rational concordance invariant.
\end{cor}
\begin{proof}
For $i=0,1$ let $(Y_i, \s_i)$ be a $\spinc$ rational homology spheres,
$K_i \subset Y_i$ null-homologous knots, such that $(Y_0,K_0, \s_0)$
and $(Y_1,K_1, \s_1)$ are rationally concordant. We want to
prove that $V_{K_0}(m, \s_0)= V_{K_1}(m, \s_1)$.

Let $(W, C, \mathfrak{t})$ be a $\spinc$ rational homology cobordism
from $(Y_0,K_0, \s_0)$ to $(Y_1,K_1, \s_1)$. Pick a suitably large $q\geq 0$, glue to $W$ a $q$-framed 2-handle along $K_0 \subset  -Y_0 \subset \partial
W$ and a $(-q)$-framed 2-handle along $K_1\subset
Y_1 \subset \partial W$. Denote by $W'$ the resulting cobordism
from $Y_q(K_0)$ to $Y_q(K_1)$. Note that $W'$ is naturally
equipped with a $\spinc$ structure $\mathfrak{t}'$ agreeing with
$\mathfrak{t}$ on $W \subset W'$, and restricts to the $\spinc$
structure $\s_m$ of Proposition \ref{correctionterms} on its two boundary components.

Clearly $W'$ is not a rational homology cobordism, as $H_2(W'; \Z)=\Z^2$ is
generated by the homology classes of the two $2$-handles we attached
along $K_0$ and $K_1$. For $i=0, 1$ let $\Delta_i \subset W'$ be the
core disk of the $2$-handle attached along $K_i$.  Set $S= \Delta_0 \cup
C \cup -\Delta_1$; then $S$ is a two-sphere with $S \cdot S=0$.
Surgering out $S$ we get a rational homology cobordism $X$ from $Y_q(K_0)$ to $Y_q(K_1)$ with $H_1(X; \Z)= \Z/q \Z$. Since by
construction $\langle c_1(\mathfrak{t}') , [S] \rangle=0$, the
$\spinc$ structure $\mathfrak{t}'|_{W' - \nu S}$ extends to a $\spinc$
structure $\mathfrak{t}_X$ of $X= (W' -\nu S) \cup_\partial S^1 \times
B^3$. Summarizing, the pair $(X, \mathfrak{t}_X)$ provides a $\spinc$
rational homology cobordism from $(Y_q(K_0), \s_m)$ to $(Y_q(K_1),
\s_m)$. This implies that $d(Y_q(K_0), \s_m)=d(Y_q(K_1),
\s_m)$, and the claim now follows from Proposition \ref{correctionterms}.
\end{proof}

\begin{cor}\label{vanishing}$V_K(m, \s) \equiv 0$ for a rationally slice knot $(Y, K, \s)$. 
\end{cor}
\begin{proof} A rationally slice knot is rationally concordant to the unknot  $(S^3, U , \mathfrak{u)}$. For the unknot one has  
\[  \CFKinfty(S^3, U , \mathfrak{u)}= \F[U,U^{-1}] \cdot z, \ \ \ \partial z =0 \]
graded so that $A(z)=j(z)=M(z)=0$. Thus, 
\begin{equation*}
V_U(m, \mathfrak{u})= d(S^3, \mathfrak{u})-2 \max (A(z)-m, j(z))=-2 \max( -m , 0)=0   ,
\end{equation*}
for every $m \geq 0$.
\end{proof}

\begin{proof}[Proof of Lemma \ref{rationalknots}] 
The proof is a simple adaptation of \cite[Proposition~11]{hom2017survey}.
Our first task is to find a cycle $\xi$ with bi-filtration level $(0,0)$. 
If $(Y, K, \s)$ is rationally slice then
\[ 0=V_K(0,\s)=d(Y, \s) -2 \cdot \min_\xi \max(A(\xi), j(\xi))= -2 \cdot \min_{\xi} \max(A(\xi), j(\xi)) \]
where the minimum is taken over all cycles $\xi \in \CFKinfty(Y, K,
\s)$ having Maslov grading $M(\xi)=d(Y, \s)=0$. Here $d(Y, \s)=0$,
since $(Y,\s)$ bounds a $\spinc$ rational homology disk.  Thus, we can find a
cycle $\xi$ with Maslov grading zero representing the generator of $H_0(\CFKinfty(Y, K, \s) )
\simeq \Z_2$ such that $\max(A(\xi), j(\xi))=0$.
 
We claim that the Alexander and algebraic filtration levels of $\xi$
are both zero.  Notice that $\xi^+$, the projection of $\xi$ on
$\CFKinfty(Y, K, \s)\{ j \geq 0\} \simeq CF^+(Y,\s)$, represents the
generator of the tower of $HF^+(Y, \s)$. If we prove that the
projection of $\xi^+$ on $\CFKinfty(Y, K, \s)\{ j \geq 0, A\geq 0\}$
is non-zero we are done. To this end, note that the
projection
\[ \pi\colon CF^+(Y, \s) =\CFKinfty(Y, K, \s)\{ j \geq 0\} \to  \CFKinfty(Y, K, \s)\{ j \geq 0,\ A\geq 0\}\]
is dual (\emph{cf.}~\cite{OS7}) to the inclusion
\[ \iota \colon \CFKinfty(-Y, K, \overline{\s})\{ j \leq 0,\ A\leq 0\} \to CF^-(-Y, \overline{\s}) =\CFKinfty(-Y, K, \overline{\s})\{ j \leq 0\} \ . \]
Since the knot $-K$, viewed now in $-Y$, is also rationally slice, we
have $V_{-K}(0, \overline{\s}) =0$. This implies that $\iota$ is surjective on the free
summand, hence $\pi$ is injective on the tower, leading to the
conclusion that $\pi(\xi^+)\not=0$. This proves that the top component
of $\xi$ has bi-filtration level $(0,0)$.

Let $\x_0, \dots, \x_m$ denotes a basis $\widehat{CFK}(Y, K, \s)=\CFKinfty(Y, K, \s)\{ j =0\}$, \text{e.g.} the one coming from the Heegaard diagram. Since $\xi$ has bi-filtration level $(0,0)$ there is a basis element, say $\x_0$, in bi-filtration level $(0,0)$ such that 
\[\xi=\x_0+(\text{lower bi-filtration level terms})\ .\]  
We complete $\xi_0=\xi$ to a $\F[U, U^{-1}]$-basis $\{\xi_0, \xi_1, \dots, \xi_m \}$ of $\CFKinfty(Y, K, \s)$ by simply taking $\xi_i=\x_i$ for $i=1, \dots, m$. Obviously:
\[\CFKinfty(Y, K, \s)= \F[U,U^{-1}] \cdot \xi \oplus \text{Span}_{\F[U, U^{-1}]}\langle  \xi_1, \dots, \xi_m \rangle \] 
as bi-filtered module, but not necessarily as chain complex (it can happen that $\xi_0=\xi$ is a component of $\partial\xi_i$ for some basis element $\xi_i$). This issue can be overcome as follows. Suppose that $\xi_0$ appears in the differential of a basis element  $\xi_i$. Since $[\xi_0]$ is non zero in $H_*(\CFKinfty(Y, K, \s)\{A\geq 0, j \geq 0\})$ we can write 
\[\partial \xi_i=\xi_0+ U^k\xi_j+\dots\ ,\]
for some $\F$-generator $U^k\xi_j$ with bi-filtration level $\geq(0,0)$. Then we substitute the basis element $\xi_j$ with $\xi_j+ U^{-k}\xi_0$. After repeating this move a few times, we can make sure that $\xi_0$ does not appear as a component of the differential of any other basis element. This leads to a  basis $\{\xi_0',  \dots , \xi'_m\}$ of $\CFKinfty(Y, K, \s)$ such that 
\begin{itemize}
\item $\CFKinfty(Y, K, \s)= \F [U, U^{-1}] \cdot \xi'_0 \oplus \text{Span}_{\F[U, U^{-1}]}\langle \xi'_1, \dots , \xi'_m \rangle$ as bi-filtered module 
\item $\xi'_0=\xi_0=\xi$, and hence $\partial \xi'_0=0$
\item $\partial \xi'_i\in \text{Span}_{\F[U, U^{-1}]}\langle \xi'_1, \dots , \xi'_m \rangle$ for $i=1, \dots, m$ .
\end{itemize}
The desired splitting is given by $\CFKinfty(Y, K, \s)= \F [U, U^{-1}] \cdot \xi \oplus \mathcal{A}$, where $\mathcal{A}=\text{Span}_{\F[U, U^{-1}]}\langle \xi'_1, \dots , \xi'_m \rangle$ 
\end{proof}

The above result can be summarized as follows:

\begin{lem}\label{main1} 
The map $\CFKinf \colon \hspace{0.1cm} \mathcal{C}_\Q \to \mathcalCFK$
  defines a group homomorphism.
\end{lem}
\begin{proof}
$\CFKinfty$ associates to a null-homologous knot $K$ in a $\spinc$
  $\Q HS^3$ $(Y,\s)$ a knot type complex $\CFKinfty(Y,
  K, \s)$ with $d(\CFKinfty(Y,K, \s))= d(Y, \s)$, well-defined up to
  $\Z \oplus \Z$-filtered chain homotopy equivalence. As a
  consequence of Theorem \ref{invariance}, the map  $\CFKinfty$ descends to a
  map $\mathcal{C}_\Q \to \mathcalCFK$. This is a group
  homomorphism in view of \cite[Proposition 4]{OS7}.
\end{proof}

We conclude this section by recalling yet another (but, as it turns out to
be, equivalent) equivalence relation among chain complexes: local 
equivalence. Although 
stable equivalence turns out to be very convenient in defining invariants
(as it will be clear from our later constructions), local equivalence
is phrased more naturally, since it takes maps between the actual
chain complexes into account, and these maps are naturally induced by 
cobordisms and concordances.

\begin{defi}
Suppose that $\mathcal{K}_i$ for $i=1,2$ are knot type chain complexes
over $\F[U, U^{-1}]$; the we say that $\mathcal{K}_1$ and
$\mathcal{K}_2$ are \textbf{locally equivalent} if there are graded,
$\Z \oplus \Z$-filtered chain maps $f\colon \mathcal{K}_1\to
\mathcal{K}_2$ and $g\colon \mathcal{K}_2\to \mathcal{K}_1$ inducing
isomorphisms on the homologies.
\end{defi}
\begin{thm}[\cite{zemke}] 
If the knots $(Y_1,K_1, \s _1)$ and $(Y_2,K_2, \s _2)$ are rationally
concordant (in the sense of Definition~\ref{ratconcordance}) then the
corresponding graded, $\Z\oplus \Z$-filtered chain complexes $\CFK ^{\infty} (Y_i,
K_i, \s _i)$ are locally equivalent. \qed
\end{thm}

\begin{prop}
The knot type chain complexes $\mathcal{K}_i$ for $(i=1,2)$ are stably
equivalent if and only if they are locally equivalent.
\end{prop}
\begin{proof}
Stable equivalence clearly implies local equivalence. Indeed, if
$f\colon \mathcal{K}_1\oplus {\mathcal {A}}_1\to \mathcal{K}_2\oplus
{\mathcal {A}}_2$ is a chain homotopy equivalence (with ${\mathcal
  {A}}_i$ acyclic complexes), then the restriction of $f$ to
$\mathcal{K}_1$, composed with the projection $pr\colon \mathcal{K}_2\oplus
{\mathcal {A}}_2\to \mathcal{K}_2$ provided a chain map inducing isomorphism on
homology.

The converse is the content of \cite[Lemma~3.3]{involutive2} when
$\mathcal{K}_2=\F [U,U^{-1}]$. The general case of this converse
direction then follows by applying \cite[Lemma~3.3]{involutive2} to
$\mathcal{K}_2^*\oplus \mathcal{K}_1$ and $\F [U,U^{-1}]$ (where
${\mathcal {K}}_2^*$ denotes the dual of ${\mathcal {K}}_2$).
\end{proof}

\section{Knot Floer invariants from cyclic branched covers}
\label{uno}
In \cite{alfieri1} the first author described a general construction
producing maps $\mathcalCFK/_\sim \to \R$ that can be used to generate
rational concordance invariants $\mathcal{C}_\Q \to \R$ by
precomposing with $\CFKinfty$.  Recall that a region of the Euclidean
plane $C \subset \R^2$ is  \textbf{south-west} if it is
closed, non-empty, not equal to $\R ^2$ and $$(\overline{x},
\overline{y}) \in C \Rightarrow \{ (x,y) \ | \ x\leq \overline{x}, y
\leq \overline{y}\} \subseteq C.$$ Given a knot type complex $\K_*$,
denote by $\K_*(C)$ the $\F [U]$-submodule of $\K_*$ spanned by the
generators with $(j,A)$-invariants lying in $C$. For a south-west
region $C$ of the plane and a knot type complex $\K_*$ set
\[ \Upsilon^C(\K_*)= \inf \left\{ t \ | \ \K_*(C_t) \hookrightarrow \K_* \text{ induces a surjective map on }H_{d(\K_*)} \right\} ,\]   
where $C_t= \{(x,y)\ | \ (x-t, y-t) \in C\}$ denotes the translate of
$C$ with the vector $v_t=(t,t)$. Here we are using the Maslov grading
as homological grading, so $H_q(\K_*)= \Z_2$ for $q\in d(\K_*) + 2
\Z$, and zero otherwise.  Since the elements with Maslov grading equal to $d(Y, \s )$
 form a finite dimensional subspace, it is easy to see that
for $t\ll 0$ the map $\K_*(C_t) \hookrightarrow \K_*$ induces the
zero-map on $H_{d(\K_*)}$, while for $t\gg 0$ the induced map is
surjective.  The same finiteness also implies that the infimum
appearing in the definition above is indeed a minimum, hence in what follows we will write $\min$ instead of $\inf$. The next result was pointed out in \cite{alfieri1}.

\begin{prop} \label{main2} Let  $C$ be a south-west 
region. If $\K_*$ and $\K'_*$ are two stably equivalent knot type
complexes then $\Upsilon^C(\K_*)=\Upsilon^C(\K'_*)$. Consequently, for
every south-west region $C$ we get a map $\Upsilon^C:
\mathcalCFK/_\sim \to \R$. \qed
\end{prop} 

In Figure~\ref{swregions} some notable families of south-west regions
are shown. Thanks to Proposition \ref{main2} each of these families
can be used to produce a one-parameter family of stable equivalence
invariants.

\begin{figure}
\includegraphics[width=6cm]{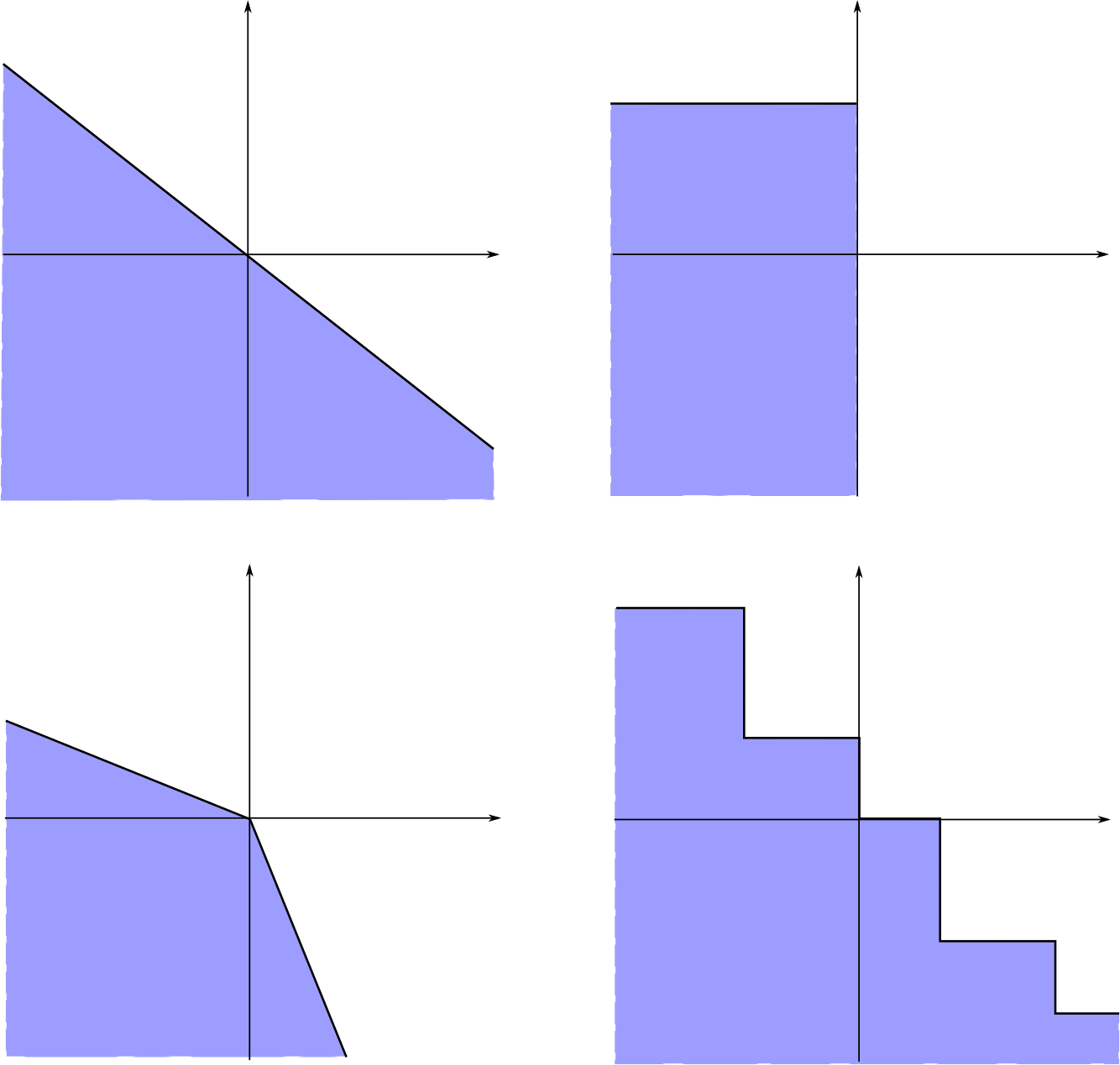}
\caption{Examples of normalized south-west domains.}
\label{swregions}
\end{figure}

In what follows, we will assume that the south-west region $C$ is
\textit{normalized}, that is
it contains the origin $(0,0) \in \R^2$ on its
boundary $\partial C$. This condition ensures that
$\Upsilon^C(\CFKinfty(S^3, U, \mathfrak{u}))=0$.

\subsection{Some concordance invariants of classical knots} \label{sec:concinvariantsclassical}
We can now use the results of the previous sections to define some
concordance invariants of knots in $S^3$. Given $K \subset S^3$ we can
form its $m$-fold cyclic branched cover $\Sigma^m(K)$. This is a
$3$-manifold equipped with an order $m$ self-diffeomorphism $\tau$.
The fixed locus of the $\Z/m \Z$-action defined by $\tau$ describes a
knot $\widetilde{K} \subset \Sigma^m(K)$. The following result is part of
knot theory folklore (see \emph{e.g.}~\cite[Section~3]{livingston2002seifert}).

\begin{lem}\label{folklore} If $m$ is a prime power, then $\Sigma^m(K)$ is a rational homology sphere.
\end{lem}
\begin{proof}
$\widetilde{K} \subset \Sigma^m(K)$ is a null-homologous knot since it
  bounds any lift $\widetilde{F} \subset \Sigma^m(K)$ of a Seifert
  surface $F \subset S^3$ of $K$. Using a Seifert matrix $
  \theta_{\widetilde{F}}$ of $\widetilde{F}$ one can define the
  Alexander polynomial of $\widetilde{K}$ via the formula $
  \Delta_{\widetilde{K}}(x)= \det ( x \theta_{\widetilde{F}}- x^{-1}
  \theta^t_{\widetilde{F}})$.

The first homology $H_1(\Sigma^m(K); \Z)$ is of order
$|\Delta_{\widetilde{K}}(-1)|$ if the latter is non-zero, and infinite
otherwise. The Alexander polynomial of $\widetilde{K}$ can be computed
from the usual Alexander polynomial of $K$ by the formula
\[ \Delta_{\widetilde{K}}(x)= \prod_{i=0}^{m-1} \Delta_K(\omega^{i} x^{1/m}) \]
where $\omega$ denotes a primitive $m^{th}$ root of unity.  When
$m=p^k$ is a power of a prime,
$$|\Delta_{\widetilde{K}}(-1)|=\prod_{i=0}^{m-1}
|\Delta_K(\omega^{i})|\not=0 ,$$ since the Alexander polynomial of a knot
does not vanish at any $p^k$ root of unity: if it did then it would
be divisible by the cyclotomic polynomial $\phi_{p^r}(x)=\sum_{i=0}^{p-1}
x^{ip^{r-1}}$ implying that $\phi_{p^r}(1)=p$ divides 
$\Delta_K(1)=1$, a contradiction. 
\end{proof}

\begin{rem}
The Alexander polynomial $\Delta_{\widetilde{K}}(x)$
admits a refinement according to spin$^c$ structures, see
\cite{Turaev}.
\end{rem}

The $m$-fold cyclic branched cover $\Sigma^m(K)$ of a knot $K\subset
S^3$ admits a distinguished  spin structure $\s_0$ that can be used
to canonically identify $\Spinc(\Sigma^m(K))$ with its second
cohomology group $H^2(\Sigma^m(K); \Z) \simeq H_1(\Sigma^m(K); \Z)$ (\emph{cf}.~\cite{grigsby2008knot} and \cite{levine2008computing}).

\begin{lem}[Lemma 2.1~\cite{grigsby2008knot}]\label{labelling} 
Let $\Sigma^m(F)$ denote the $m$-fold cyclic branched cover of a
properly embedded surface $F \subset B^4$, with boundary a knot $K
\subset S^3$. Denote by $\widetilde{F} \subset \Sigma^m(F)$ the fixed
locus of the covering action of $\Sigma^m(F)$. Then there is a unique
spin structure $\mathfrak{t}_0$ on $\Sigma^m(F)$ characterized as
follows:
\begin{itemize}
\item if $m$ is odd, the restriction of $\mathfrak{t}_0$ to $\Sigma^m(F)- \nu \widetilde{F}$ is the pull-back $\widetilde{\mathfrak{t}}$ of the spin structure of $B^4- \nu F$ extending over $B^4$,   
\item if $m$ is even, the restriction of $\mathfrak{t}_0$ to $\Sigma^m(F)- \nu \widetilde{F}$ is $\widetilde{\mathfrak{t}}$ twisted by the element of $H^1(\Sigma^m(F)- \nu \widetilde{F}; \Z_2)$ supported on the linking circle of $\widetilde{F}$. 
\end{itemize}
\end{lem}

We define the canonical spin structure $\s_0$ of the three-manifold
$\Sigma^m(K)$ as the restriction of the spin structure $\mathfrak{t}_0
\in \Spinc(\Sigma^m(F))$ of Lemma \ref{labelling} to $\Sigma^m(K)=
\partial( \Sigma^m(F))$, where $F \subset B^4$ denotes a pushed-in
Seifert surface of $K \subset S^3$.

\begin{lem}\label{concordancecovers} 
Let $\mathcal{C}_\Q^{\Spin}$ denote the subgroup of $\mathcal{C}_\Q$
spanned by the triples $(Y, K, \s)$ with $\s$ spin. Let $m=p^r$ be a
power of a prime. Then the map $\Sigma^m \colon  K \mapsto (\Sigma^m(K),
\widetilde{K}, \s_0)$ descends to a group homomorphism $\Sigma^m:
\mathcal{C} \to \mathcal{C}_\Q^\Spin$.
\end{lem}
\begin{proof}
Since the $m$-fold cyclic branched cover of the two-sphere branched on
two points is the two-sphere, we conclude that the map $\Sigma^m$ respects
connected sums; moreover, the spin structures on the summands glue to the spin structure on the connected sum.
Thus, it is enough to prove that if $K$ is a slice knot
then $(\Sigma^m(K), \widetilde{K}, \s_0)$ is rationally slice.

Suppose that $K \subset S^3$ bounds a smooth disk $\Delta \subset
B^4$. By taking the $m$-fold cyclic branched cover of $\Delta$ we get
a rational homology ball $\Sigma^m(\Delta)$. In $\Sigma^m(\Delta)$ the
knot $\widetilde{K} \subset \partial \Sigma^m(\Delta)$ bounds a smooth
disk $\widetilde{\Delta}$, namely the fixed locus of the covering
action of $\Sigma^m(F)$. Furthermore, the spin structure
$\mathfrak{t}_0$ of Lemma \ref{labelling} provides a spin extension of
$\s_0$ to $\Sigma^m(F)$.
\end{proof}

\begin{thm}\label{theorem1} Let $C \subset \R^2$ be a south-west region, and $m=p^r$ a power of a prime. Given a knot $K \subset S^3$ set $\Upsilon^C_m(K)= \Upsilon^C( \CFKinfty(\Sigma^m(K), \widetilde{K}, \s_0))$. Then $\Upsilon^C_m(K)$ is a knot concordance invariant. 
\end{thm}
\begin{proof} The result follows by combining Lemmas~\ref{main1}, \ref{concordancecovers} and Proposition \ref{main2}.
\end{proof}

\begin{rem} Using south-west regions which are symmetric with respect to the $x=y$ axis, one can construct further invariants $\overline{\Upsilon}^C_m(K)$ and $\underline{\Upsilon}^C_m(K)$ by means of the involutive Heegaard Floer homology \cite{manoleacuendriks}. For a related definition of 
involutive upsilon invariants, see \cite{HL}. 
\end{rem}

For a given cohomology class $\xi \in
H^2(\Sigma^m(K); \Z) \simeq H_1(\Sigma^m(K); \Z)$ we can set
\[\Upsilon^C_\xi(K)= \Upsilon^C( \CFKinfty(\Sigma^m(K), \widetilde{K}, \s_0+\xi)) \ , \]
where $\s_0+\xi $ denotes the spin$^c$ structure we get from $\s _0$ by twisting it with $\xi$. 
The following result is a straightforward adaptation of~\cite[Theorem 1.1]{grigsby2008knot}, and provides a way to obtain some additional concordance information
 from the construction above.

\begin{thm}\label{theorem3} 
Let $K \subset S^3$ be a knot. 
Suppose that $m=p^r$ is a power of a prime and 
set 
\[{\det}_m(K)= \prod_{i=0}^{m-1} |\Delta_K(\omega^{i})| \]
where $\omega$ is a primitive $m$-th root of unity. 
If $K$ is a slice knot then there
exists a subgroup $G < H^2(\Sigma^m(K); \Z)$ of cardinality
$\sqrt{\det_m(K)}$ such that $\Upsilon^C_\xi(K)=0$ for all
$\xi\in G$.  
\end{thm}
\begin{proof} 
By the previous discussion, if $K$ is a slice knot, then $\det_m(K)$ is a square. Indeed,
$\det_m(K)$ computes the cardinality of
$H_1(\Sigma^m(K); \Z)$, which is necessarily a square whenever $\Sigma^m(K)$ bounds a rational homology ball.

If $K$ is a slice knot, then $\Sigma^m(K)$ bounds a rational homology
ball $W$ to which the spin structure $\s_0$ of Lemma \ref{labelling}
extends. If $G < H^2(\Sigma^m(K); \Z)$ denotes the image of the
restriction map $\delta \colon H^2(W; \Z) \to H^2(\partial W; \Z)$ of the
long exact sequence for the pair $(W, \partial W)$, then $\s_0+\xi$
extends to a $\Spinc$ structure of $W$ for each $\xi \in G$. More
precisely, if $\mathfrak{t}_0$ denotes the $\spinc$ structure
extending $\s_0$ then $\partial (W, \mathfrak{t}_0+
\xi')=(\Sigma^m(K), \s_0+ \xi)$ where $\xi' \in H^2(W; \Z)$ is a class
such that $\delta(\xi')=\xi$.

Summarizing, whenever $K$ is a slice knot, for each $\xi \in G= \text{Im}(H^2(W; \Z) \to H^2(\partial W; \Z))$ the triple $(\Sigma^m(K), \widetilde{K}, \s_0 + \xi)$ represents a $\spinc$ rationally slice knot. Thus, as consequence of Proposition \ref{main2}
\[ \Upsilon^C_\xi(K)= \Upsilon^C(\CFKinfty(\Sigma^m(K), \widetilde{K}, \s_0 + \xi) = \Upsilon^C(\CFKinfty(S^3, U,\mathfrak{u}))=0 \ , \]
where the last identity follows from the normalization condition on the south-west
regions. A careful check with the long exact sequence
of the pair reveals that indeed $|G|=\sqrt{|H_1(\Sigma^m(K);  \Z)|}=\sqrt{\det_m(K)}$, proving the claim.
\end{proof}

\section{Homomorphisms from the knot concordance group}
\label{sec:homomorphisms}

The Upsilon-function $\Upsilon _K$ of \cite{OSS4} can be defined along 
the same ideas, by using special south-west regions. (This reformulation
of the invariants is due to Livingston~\cite{Livingston1}.)
Indeed, for $t\in [0,2]$ let us consider the half-plane
\[
H_t=\left\{ (x,y)\in \R ^2\middle| y\leq \frac{t}{t-2}x\right\} .
\]
Since $H_t$ (for the chosen $t\in [0,2]$) is a south-west region,
we get an invariant $\Upsilon^{H_t}(\K_*)$ for any knot type chain complex $\K_*$.
Define  $\Upsilon_t(\K_*)=-2 \cdot \Upsilon^{H_t}(\K_*)$.
The proof of \cite[Theorem~7.2]{Livingston1} readily generalizes to 
this setting, verifying the following proposition.

\begin{prop}\label{sums} $\Upsilon_t \colon \CFK/_\sim \to \R$ is a group homomorphism for each $t \in [0,2]$. \qed
\end{prop}


Lemma~\ref{main1}, Theorem~\ref{concordancecovers} and
Proposition~\ref{sums} then imply:

\begin{thm} Let $m=p^r$ be a prime power. For a knot $K \subset S^3$ set 
$\Upsilon_{K,m}(t) = \Upsilon_t(\CFKinfty(\Sigma^m(K), \widetilde{K}, \s_0))$. Then the map $K \mapsto \Upsilon_{K,m}(t)$ descends to a group homomorphism $\mathcal{C} \to C^0[0,2]$. \QEDB
\end{thm}

One can also form a spin$^c$ refined versions of these invariants.

\begin{thm} \label{twistedinvariants}Let $m=p^r$ be a prime power. For a knot $K \subset S^3$, and a cohomology class $\xi \in H^2(\Sigma^m(K); \Z)$ set 
\[\Upsilon_{K,\xi}(t) = \Upsilon_t(\CFKinfty(\Sigma^m(K), \widetilde{K}, \s_0+\xi)) \ .\]
If $K$ and $K'$ are concordant then there exists a subgroup $G \subset
H^2(\Sigma^m(K); \Z) \times H^2(\Sigma^m(K'); \Z)$ with $|G|=
\sqrt{\det_m(K)\cdot \det_m(K')}$ such that
$\Upsilon_{K,\xi}(t)=\Upsilon_{K',\xi'}(t)$ for each $(\xi, \xi') \in
G$. \QEDB
\end{thm}

The following theorem gives obstructions to finite concordance order. 

\begin{thm}
Let $K\subset S^3$ be a knot and $p$ a prime.  Let $\mathcal{H}_p$
denote the set of subgroups of $H^2(\Sigma^2(K);\Z)$ of order $p$.
Set 
\[\mathcal{S}_t(K,p)= \min \left\{\left| \sum_{H \in
    \mathcal{H}_p} n_H \sum_{\xi \in H} \Upsilon_{K,\xi}(t)\right| \colon 
  n_H \in \Z_{\geq 0}, \ \sum_{H \in \mathcal{H}_p} n_H >0\right\}
  \ ,\] 
If $K$ has finite order in the smooth concordance group then
  $\mathcal{S}_t(K,p) \equiv 0$.
\end{thm}
\begin{proof}
A similar statement already appeared in \cite{grigsby2008knot}, we just follow their argument.
If $K\# \dots \# K=0$ in the knot concordance group then, for some $n >0$, there is a subgroup $G \subset H^2(\Sigma^2(K);\Z)\times \dots \times H^2(\Sigma^2(K);\Z)$ with $|G|= \det(K)^{\frac{n}{2}}$ such that 
\[ 0=\Upsilon_{K\# \dots \# K, \xi}(t)= \sum_{i=1}^n \Upsilon_{K, \xi_i}(t)\] 
for each $\xi= (\xi_1, \dots, \xi_n)\in G$. We used the
identity $\Upsilon_{K\#K',\xi+\xi'}(t)=
\Upsilon_{K,\xi}(t)+\Upsilon_{K',\xi'}(t)$, which follows from \cite[Section 7]{OS7} and Proposition \ref{sums}.

Given a prime $p$, if $p$ does not divide $\det(K)=
|H^2(\Sigma^2(K);\Z)|$ then by Lagrange's theorem
$\mathcal{H}_p=\emptyset$.  If $p$ does divide $\det (K)$,then it also
divides $|G|$ and Cauchy's theorem guarantees the existence of an
order $p$ element $\xi= ( \xi_1, \dots, \xi_n)$ in $G$. Thus
\[  0=
\sum_{j=0}^{p-1}\Upsilon_{K\# \dots \# K, j \cdot \xi}(t)= 
\sum_{j=0}^{p-1}\sum_{i=1}^n \Upsilon_{K, j \cdot \xi_i}(t)=
\sum_{i=1}^n \sum_{j=0}^{p-1} \Upsilon_{K, j \cdot \xi_i}(t)= 
\sum_{i=1}^n \sum_{\xi \in H_i} \Upsilon_{K, \xi}(t) \ , \]
where $H_i \subset H^2(\Sigma^2(K);\Z)$ denotes the subgroup generated by $\xi_i$ in $H^2(\Sigma^2(K);\Z)$.
\end{proof}

\section{Alternating torus knots}
\label{sec:AltTorus}

Alternating torus knots are those of the form $T_{2,p}$ for some odd
integer $p=2n+1$. The branched double cover of $T_{2,p}$ is the lens
space $\lu$. In what follows, the lift of $T_{2,p}$ along the branched
covering $\pi \colon L(p,1) \to S^3$ will be denoted by
$\widetilde{T}_{2,p}$.

Let $K_{p,q}$ denotes the $2$-bridge knot corresponding to $\frac{p}{q}$ ( with $p>q>0$, $p$ odd). Notice that the alternating torus knot $T_{2,2n+1}$ is isotopic to the $2$-bridge knot $K_{p,q}$ with parameters $p = 2n+1$ and $q=1$.
In \cite{grigsby2006knot} Grigsby proved the existence\footnote{The actual isomorphism is between the ``hat'' versions of $HFK$; we will address the suitable extension in Section \ref{sec:combinatorial}.} of a graded quasi-isomorphism
\begin{equation}\label{eqn:griso}
L \colon \widehat{CFK}(K_{p,q}) \longrightarrow \widehat{CFK}(L(p,q), \widetilde{K}_{p,q}, \s_0), 
\end{equation}  
where $\s_0$ denotes the only spin structure of $L(p,q)$.

In this section we present a proof of Theorem \ref{thm:interi} based
on combinatorial knot Floer homology (more precisely, grid homology,
extended to lens spaces).  As suggested by an anonymous referee, this
result can alternatively be obtained as a consequence of the work of
J.~Rasmussen and S.~Rasmussen \cite[Lemma~3.2]{rasmussens}.

\subsection{Combinatorial Knot Floer homology}\label{sec:combinatorial}
In \cite{manolescu2009combinatorial} Manolescu, Ozsv\'ath and Sarkar
gave a combinatorial description of $HFK^\circ$ for knots in the
$3$-sphere using grid diagrams, cf. also~\cite{ozsvath2015grid}. This
construction has been extended by Baker, Grigsby and Hedden to the
case of knots in lens spaces. We will recall their definitions in the
specific relevant cases, pointing the interested reader to their paper
\cite{baker2008grid}.

Given a lens space $L(p,q) = S^3_{-\frac{p}{q}} (\bigcirc)$ with $p>q$ and $p$ odd, there is a
distinguished choice for the affine isomorphism between $Spin^c (\lp)$
and $H_1(\lp;\Z)$, when the unique spin structure $\s_0$ is identified
with the trivial homology class. Therefore in what follows we will
adopt the convention of labelling $\s \in \Spinc(\lp)$ by a number $h
\in \{-\frac{p-1}{2}, \ldots , \frac{p-1}{2}\}$; to such $h$ we
associate the $\spinc$ structure corresponding to $h \in H_1(\lp;\Z)
\cong \Z_p$. 
Notice that with the conventions above we have that $\overline{\s_0 + h} = \s_0 -h$, where the bar denotes the conjugation on $\spinc$ structures.  

A \emph{twisted grid diagram} for $\lu$ is a planar representation of
a genus one Heegaard diagram for $\lu$, produced by the minimally
intersecting one by doubling both curves. We choose to draw this
diagram as shown in Figure \ref{fig:griglia5-2}, where each little
square has edges of length one. This is a $2\times 2p$ grid, with a
specific identification of its boundary: the left and right edges of
the grid are identified as $(0,t) \sim (2p,t)$ with $t \in [0,2]$,
while the top part is identified with the bottom after a shift taking
$(x,2)$ to $(x-2,0)$ $\modd{2p}$. The integer $n$ is called the \emph{dimension} of $G_p$.

By placing in the squares two sets of markings $\XX = \{ X_0,X_1\} $ and $\OO =\{ O_0,O_1\}$,
we get a twisted grid diagram $G_p$, which determines a knot $K$ in $\lu$. 
It was shown in \cite{grigsby2006combinatorial} that if we are
considering the lift of a $2$-bridge knot to $\lp$, then the resulting
knot can be presented as above.

In the case of the lift of $T_{2,2n+1}$ to $L(2n+1,1)$, we get a
$2\times 2(2n+1)$ grid with $\OO = \{ 0,1\} $ and $\XX = \{
2n+1,2n+2\} $, where the  notation means that we are placing the first
$\OO$ marking in the center of the $0$-th square from the left in the
bottom-most row of $G_p$, and the second $\OO$ in the first box on the
top row, and likewise for the $\XX$s. One example is shown in Figure
\ref{fig:griglia5-2}.
\begin{figure}
\includegraphics[width=7cm]{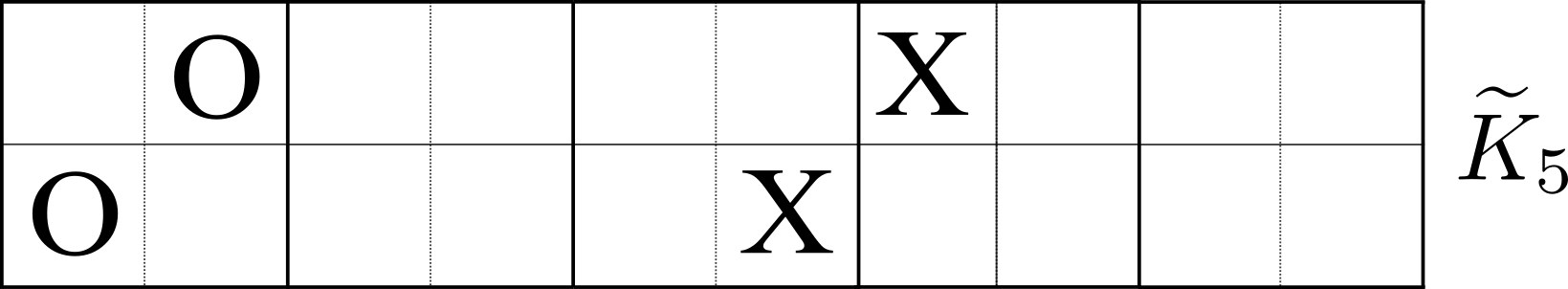}
\caption{The grid $G_5$ for the lift of $T_{5,2}$ in $L(5,1)$.}
\label{fig:griglia5-2}
\end{figure}

Recall that the generators of the knot Floer complex are given by intersections of the
$\alpha$- and $\beta$-curves determining the Heegaard splitting. Here
the generating set is given by bijections between the horizontal and
vertical circles (called
$\alpha_0$, $\alpha_1$ and $\beta_0, \beta_1$ respectively) in the grid, after the identifications. We require that the two points of a generator lie on different curves.

Since each pair $(\alpha_i, \beta_j)$ intersect exactly $p$
times, it is easy to show that the generating set $S(G)$ is in
bijection with $\mathfrak{S}_2 \times \Z_p^2$. Generators
corresponding to $ Id_{\mathfrak{S}_2} \times (a_0, a_1) $ will be
denoted by $x_{a_0,a_1}$, while those of the form $(12) \times (a_0,
a_1)$ by $y_{a_0,a_1}$. The two pairs $(a,b) \in \Z_p^2$ will be
called the \emph{$p$-coordinates} of the generator.

Given such a multi-pointed grid $G_p$, representing the pair $(\lu,
\widetilde{T}_{2,p})$, the complex $GC^\circ (G_p)$ is the free
$R^\circ$-module generated by $S(G_p)$. The specific choice of the
base ring, together with the choice of a differential, will determine
the flavour of grid homology we will consider. In what follows we will
restrict to $R^\infty = \F [V_0^{\pm 1},V_1^{\pm 1}]$ or $R^- = \F
[V_0,V_1]$, corresponding respectively to the $\infty$ and
minus versions of grid homology.  The variables $V_i$ act on
the complex, and can be thus thought of as graded endomorphisms;
they correspond to the $\OO$-markings.

%

Both the Alexander and Maslov degrees $M(x), A(x) \in \Q$  of a generator $x \in S(G_p)$
can be combinatorially defined from the grid as explained in
\cite{grigsby2006combinatorial}.
(Although the $A$-grading is in
general $\Q$-valued, but $A(x)\in \Z$ for all generators if the knot
$K$ is null-homologous, which is always the case for lifts of knots to their cyclic covers.)  
Finally we can associate to $x$ a
$\spinc$ structure $\s(x) \in \Spinc(\lu) \cong \Z_p$ following the
recipe of \cite[Section~2.2]{baker2008grid}: if the $p$-coordinates of
$x$ are $(a,b)$, then $\s(x) \equiv a + b \modd{p}$.

Each $V_i$ decreases the Maslov degree by two, the Alexander degree by one, and does not
change the $\spinc$ structure. If $V_0^r V_1^s x \in \F[V^{\pm 1}_0,V^{\pm 1}_1]\langle S(G_p)\rangle$, 
then the number $-r - s$ of $V_i$ variables is called the \emph{algebraic filtration}, and is denoted by $j(V_0^r V_1^s x )$.

The differential for the $\infty$ flavour is the
$\F[V_0^{\pm1},V_1^{\pm1}]$-chain map defined as follows:
\begin{equation}\label{eqn:differential}
\partial^\infty (x) = \sum_{y \in S(G)} \sum_{\substack{r \in Rect^\circ (x,y) \\ r \cap \XX = \emptyset} }  V_0^{O_0(r)} V_1^{O_1(r)} \cdot y
\end{equation}
In Equation \eqref{eqn:differential}, $Rect^\circ (x,y)$ denotes the
set of (oriented) empty rectangles $r$ embedded in the grid, such that
the edges of $r$ are alternatively on $\alpha$ and $\beta$ curves, and
one component of $x$ is in the lower-left vertex of $r$; a rectangle
is \emph{empty} if $\mathring{r}$ does not contain any component of
$x$ or $y$. Note that in the expression \eqref{eqn:differential} we
are also requiring trivial intersection between rectangles and $\XX$
markings.

To obtain the minus flavour, keep the same differential and restrict
the base ring to $\F[V_0,V_1]$.  If we impose $V_1 = 0$ at the complex
level, we obtain a complex computing the hat version of knot Floer
homology. These are finitely generated modules, over $\F[U^{\pm 1}]$,
$\F[U]$ and $\F$ respectively. Here $U$ is the map induced in homology
by either of the $V_i$s (and it is not hard to see that multiplication
by $V_0$ and $V_1$ coincide in homology).

\begin{rem}
If we fix an intersection point between one $\alpha$ and one $\beta$ curve, and $i \in \Z_p$, there is only one generator $x \in S(G_p)$ having that intersection as a component, and such that $\s(x) = i$.
\end{rem}

\begin{rem}
As alluded in the introduction, the quasi-isomorphisms from \cite{grigsby2006knot} $$L \colon
\widehat{CFK}(K_{p,q}) \longrightarrow \widehat{CFK} (L(p,q),
\widetilde{K}_{p,q}, \s_0)$$ induce quasi-isomorphisms on the other
flavours (\emph{i.e.}~the filtered $-$ and $ \infty$ versions) as
well. This follows from the fact that $L$ defined at the complex level
is a grading preserving isomorphism of bi-graded vector spaces, and in
this case both complexes have trivial differential. This implies that
there can only be horizontal and vertical differentials, dictated by
the spectral sequence to $\widehat{HF}$ of the underlying manifold
(which is an $L$-space in both the domain and codomain of
$L$).
\end{rem}

The purpose of the next section will be to use this combinatorial
description for a recursive computation of $GC^\infty(L(2n+1,1),
\widetilde{T}_{2n+1,2},\s_0 + h)$. Recall that thanks to
\cite[Theorem~1.1]{baker2008grid}, we know that this complex is in
fact chain homotopic to its holomorphic counterpart
$\CFKinfty(L(2n+1,1), \widetilde{T}_{2n+1,2},\s_0 + h)$.

\subsection{Computations}\label{sec:computation}
If $p = 2n+1$ is an odd integer, and $h \in \{0, \ldots , n\}$ we want to prove the existence of the following quasi-isomorphisms:
\begin{equation}\label{eqn:lp1}
CFK^\circ (L(p,1), \widetilde{T}_{2,p}, \s_0 \pm h) \simeq CFK^\circ (T_{2,p-2h}) .
\end{equation}
Note that there will be a shift in the Maslov grading, given by the
difference in the correction terms, which can be computed using the
recursive formula in \cite[Section~4.1]{ozsvath2003absolutely}.

We can prove the existence of the isomorphisms in Equation
\eqref{eqn:lp1} by constructing graded quasi-isomorphisms
\begin{equation*}
\widetilde{F}_{p,h} \colon CFK^\circ (L(p,1),\widetilde{T}_{2,p}, \s_0 + h) \longrightarrow CFK^\circ(L(p-2h,1),\widetilde{T}_{2,p -2h}, \s_0 ), 
\end{equation*}
and post composing with the suitable inverse of the isomorphism $L$. Notice that we can write $\widetilde{F}_{p,h}$ as the composition of some more elementary maps
\begin{equation}
F_{p,h} \colon CFK^\circ (L(p,1),\widetilde{T}_{2,p},\s_0 +h) \longrightarrow CFK^\circ (L(p-2,1),\widetilde{T}_{2,p-2},\s_0 +h -1). 
\end{equation}
shifting the $\spinc$ structure by one and decreasing $p$ by two, as shown in Figure \ref{fig:mappeliftinteri}. We can define the analogous maps $\overline{F}_{p,h}$ by precomposing $F_{p,h}$ with conjugation on $\spinc$ structures.

The grid complex $GC^\circ (L(p-2,1),\widetilde{T}_{2,p-2}, \s_0 + h -1)$ has four generators less than $GC^\circ (L(p,1),\widetilde{T}_{2,p}, \s_0 + h)$. We will choose some generators to be removed from the complex associated to $\widetilde{T}_{2,p}$ in such a way that we are cancelling precisely the four generators that comprise an acyclic subcomplex, and their cancellation induces the needed quasi-isomorphism.

\begin{figure}[h!]
\includegraphics[width=10cm]{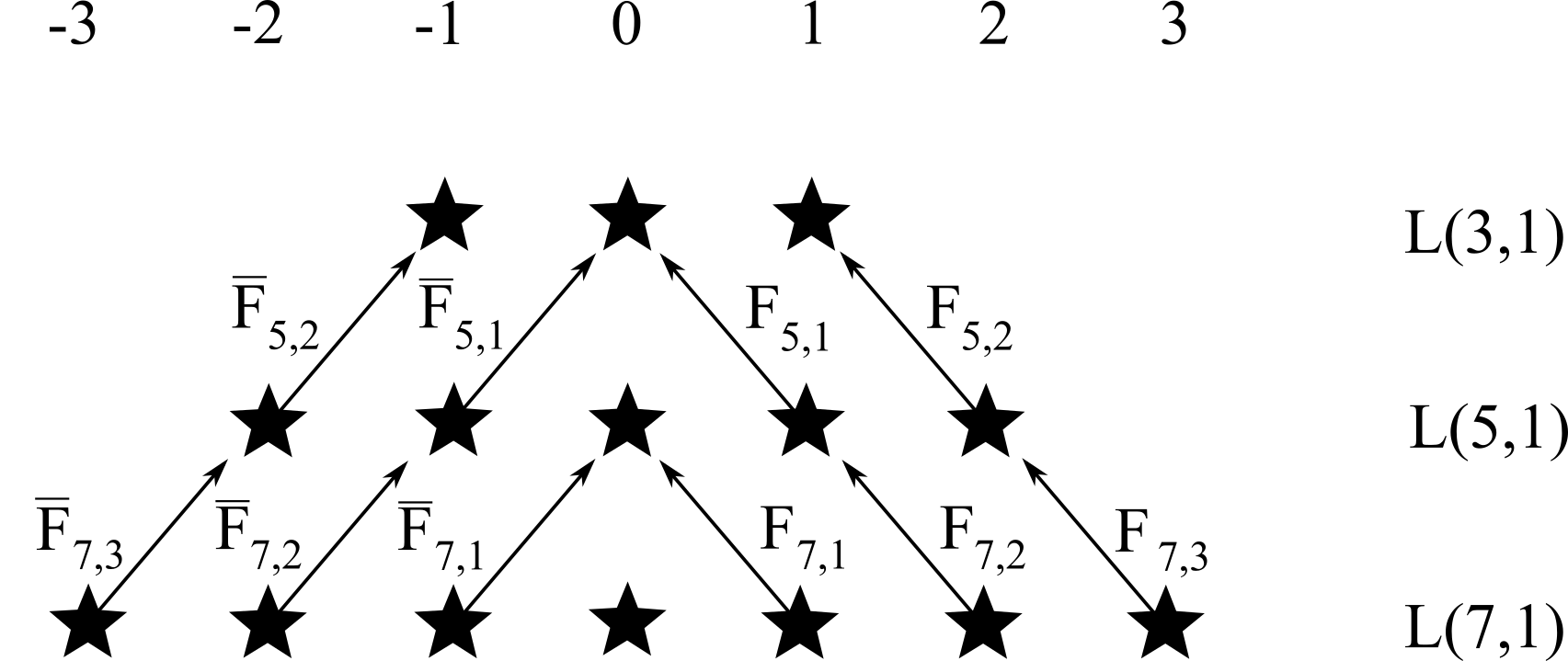}
\caption{How the maps $F_{p,h}$ fit together. The horizontal label denotes the $\spinc$ structures in the lens space written on the right. Stars represent the corresponding knot Floer homology groups.}
\label{fig:mappeliftinteri}
\end{figure}

Let us introduce some objects that will come into play in the proof of
Theorem \ref{thm:interi}. These will be some model complexes,
quasi-isomorphic to $GC^\circ (L(p,1),\widetilde{T}_{2,p}, \s_0 + h )$
in a given Alexander degree.

An \emph{electric pole} of length $e \ge 0$ is the graded complex over
$\F$ described in Figure \ref{fig:palo}. It consists of $2e + 1$
generators and $4e-2$ differentials, denoted as dots and arrows
respectively. If $e=0$ the pole consists of 
a single generator. All generators in a pole have the same
Alexander degrees.

\begin{figure}
\includegraphics[width=7cm]{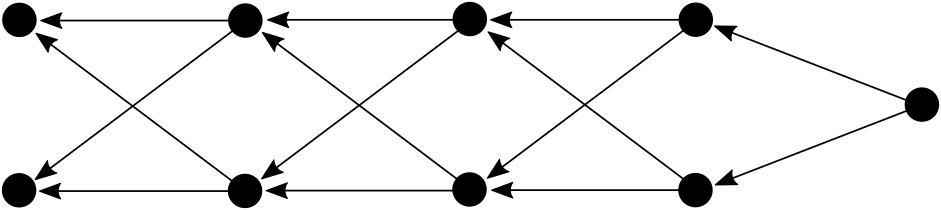}
\caption{An electric pole of length $4$. Each dot represents a
  generator over $\F$ in the $(j,A)$-plane, and arrows (from right to
  left) denote differentials. Edges represent differentials not
  involving multiplication by any $V_i$. All generators have the same
  Alexander degree.}
\label{fig:palo}
\end{figure}
A \emph{wire} of length $w>0$ is a graded complex composed by $2w+2$
generators and $4w$ differentials, as shown in Figure
\ref{fig:cavi}. The differentials in this case carry a label,
specifying if they are acting as multiplication by $V_0$ or $V_1$.

\begin{figure}
\includegraphics[width=7cm]{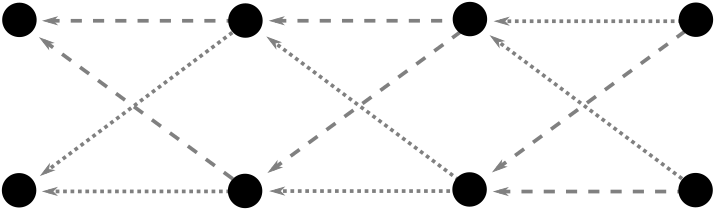}
\caption{A wire of length $3$. Dotted arrows are multiplication by $V_0$, while dashed by $V_1$.}
\label{fig:cavi}
\end{figure}

We can form a new complex by ``fusing'' together two electric poles
(of the same length) and a wire, obtaining a complex whose homology
has rank $1$, generated by the circled generator in Figure
\ref{fig:complessogenerale}.  More formally, consider two electric
poles of length $e$ and a wire of length $w$, and identify the two
leftmost generators of one pole with the rightmost of the wire, then
identify the two leftmost generators of the other pole with the
leftmost of the wire, as depicted in Figure
\ref{fig:complessogenerale}. We denote this complex as $C[e,w]$.  If
$e = 0$, each pole is composed by a single generator, and the two
leftmost arrows in the wire converge on the left one, while the
rightmost are the differentials of the right one, as shown in the
bottom of Figure \ref{fig:complessogenerale}.  Note that $C[e,w]$ has
$2w + 4e$ generators. We will now prove that the graded complex
$GC^\infty (G_p,\s_0 + h)$ can be built from unions of wires and
poles.

\begin{figure}
\includegraphics[width=8cm]{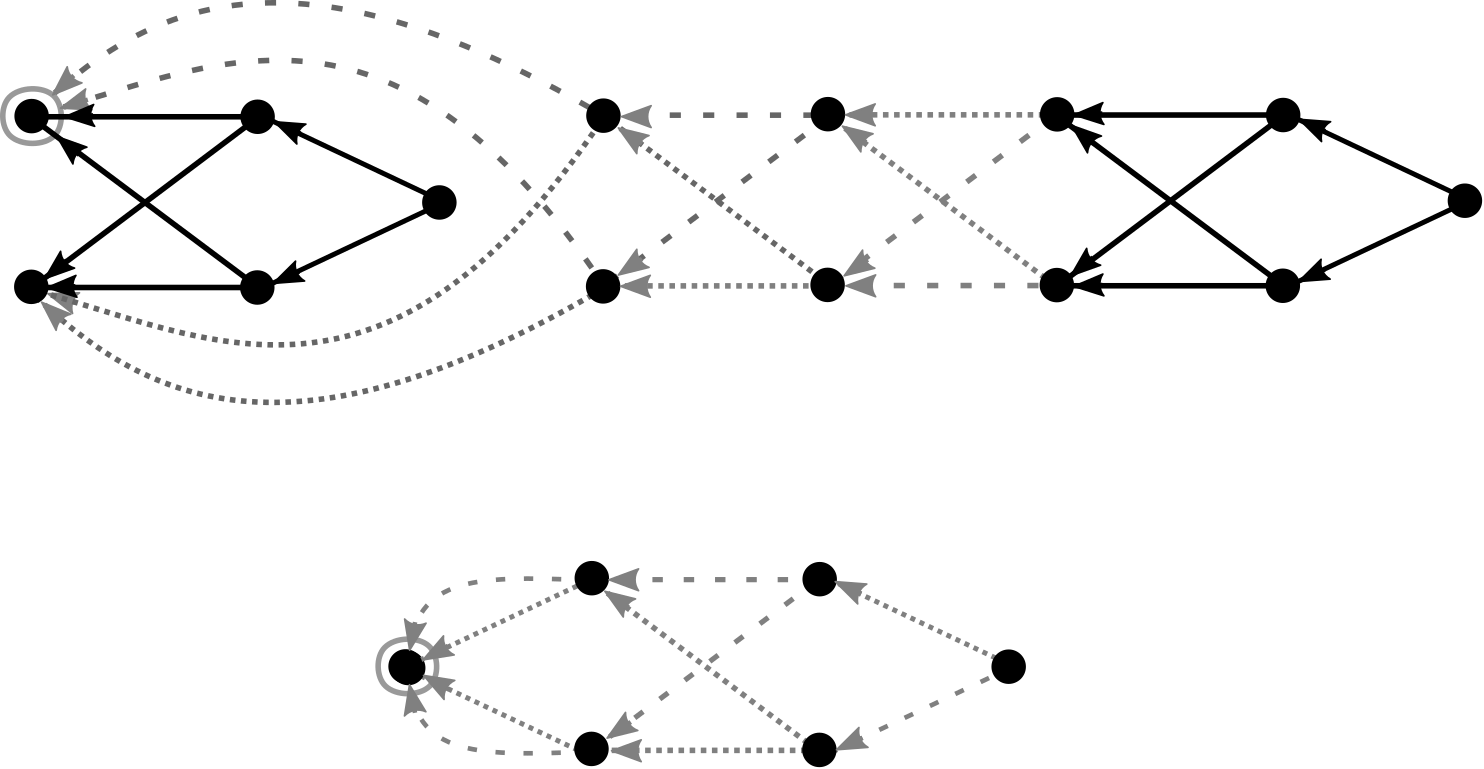}
\caption{In the top part the base complex $C[2,3]$. Black edges denote differentials not involving $V_i$ variables, dotted edges multiplication by $V_0$, and dashed by $V_1$. The circled generator is the only survivor in homology. In the lower part of the Figure, the case $e=0$ and $w=1$ (corresponding to the lift of the trefoil in its spin structure).}
\label{fig:complessogenerale}
\end{figure}

\begin{prop}\label{prop:qisom}
The complex $GC^\infty (L(p,1),\widetilde{T}_{2,p}, \s_0 \pm h)$ is quasi-isomorphic to $C[h,p -2h] \otimes \Z_2[V_0^{\pm 1},V_1^{\pm 1}]$, which in turn is quasi-isomorphic to $\CFKinfty (T_{2,p-2h})$.
\end{prop}
This means that the height of the poles is controlled by the ``distance''
from the spin structure $\s_0$; in the following proof we are going
to show how to shorten each pole without changing the homology, hence
proving that the complex is completely determined by the length of its
wires.
\begin{proof}
As before, let us denote by $G_p$ the grid for the lift of $T_{2,p}$
to $\lu$.  We first need to argue that each generator can have at most $2$ differentials
emanating from it.
\begin{lem}\label{lemma:massimodue}
 If $x$ denotes a generator in $S(G_p)$, then there are at most $2$ differentials emanating from $x$.
\end{lem}
The proof of the Lemma is deferred to the end of this subsection.  We
will subdivide the generators of $S(G_p)$ in different families, and
give a general formula for the differential on each.  First of all
note that, since markings of either kind are grouped together, they
act as ``walls'' for rectangles: if the two components of a generator
$x$ are intertwined with the $\XX$ markings, then every empty
rectangle starting from $x$ must intersect an $\OO$ marking, resulting
in the multiplication by a $V_i$ variable. On the other hand, if the
two components are on the same side with respect to some type of
marking, the differential of $x$ will be composed by elements with the
same property (and there is not going to be any multiplication by
$V_i$). This will not hold for $4$ sporadic cases, which will provide
the ``attachment'' between wires and poles.

Let $A_p$ denote the set of generators in $S(G_p)$ whose components do
not intertwine with the markings, and $B_p$ the remaining ones.  Apart
from $4$ special cases treated separately, we will prove that each generator $x \in
A_p$ has differential of the form $\partial(x) = x^\prime + x^{\prime
  \prime}$, for distinct $x^\prime , x^{\prime \prime} \in A_p$, while
if $x \in B_p$, then either $\partial(x) = V_i (x^\prime + x^{\prime
  \prime})$ or $\partial(x) = V_i x^\prime + V_{i+1}x^{\prime \prime}$
for some $i$ (indices modulo $2$). In the remaining sporadic cases, we
will also explicitly determine the differential.

Note that we know the number of generators in each $\spinc$ structure,
and we will determine explicitly the differential; since we are
considering the graded version of $GC^\infty$, and for a fixed
Alexander degree the graph induced by the differential is connected,
we only need to determine the degree of one element in order to
determine the whole complex $GC^\infty(\lu,\widetilde{T}_{2,p} ,\s)$.

\begin{figure}
\includegraphics[width=8cm]{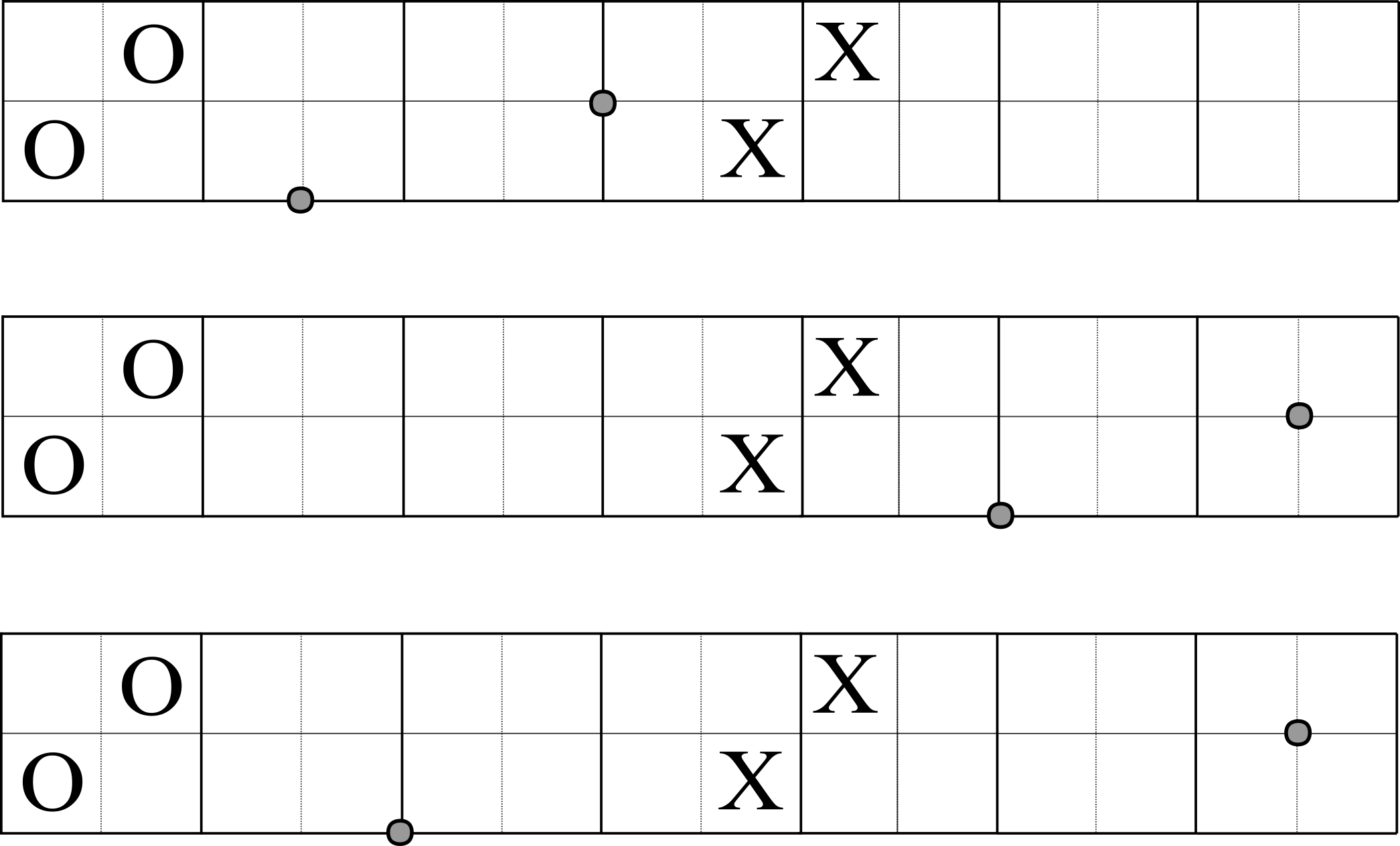}
\caption{In the top, middle and bottom part respectively we have two instances of case $A$ and one of $B$. The two gray dots indicate the components of the generator.}
\label{fig:casidim}
\end{figure}
In all cases it is easy to find two generators connected to $x$ by
empty rectangles not intersecting the $\XX$ markings. By Lemma
\ref{lemma:massimodue} we know that these are the only ones.

$\partial (x_{a,b}) = 
\begin{cases}
y_{b,a}+ y_{a-1,b+1} &  \mbox{if } x \in A_p, a\le b\\ 
y_{a,b}+ y_{b-1,a+1} &  \mbox{if } x \in A_p, b\le a\\
V_0 y_{a,b}+ V_1 y_{b-1,a+1} &  \mbox{if } x \in B_p, a\le b\\
V_0  y_{b,a}+ V_1 y_{a-1,b+1} &  \mbox{if } x \in B_p, b\le a\\
\end{cases}\\$

$\partial (y_{a,b}) = 
\begin{cases}
x_{b,a}+ x_{a,b} &  \mbox{if } y \in A_p, a\le b\\ 
x_{a+1,b-1}+ x_{b-1,a+1} &  \mbox{if } y \in A_p, b\le a\\
V_1 (x_{a-1,b+1}+  x_{b-1,a+1}) &  \mbox{if } y \in B_p, a\le b\\
V_0 ( x_{a,b}+ x_{b,a}) &  \mbox{if } y \in B_p, b\le a\\
\end{cases}$\\

For each $\spinc$ structure $\s_0 \pm h$, with $h \neq 0$, we can find
$4$ generators with a slightly different behaviour. These are the
generators such that one component is positioned on the lower left
vertex of a square containing an $\XX$ or an $\OO$ marking. We can write
them as $y_{\frac{p-1}{2}, *},y_{*^\prime,\frac{p+1}{2}}$ in the case
of $\XX$ markings, and $x_{0,*}, x_{*^\prime , 0}$ for the $\OO$'s. The
values of $*$ and $*^\prime$ are uniquely determined by $h$.  Their
differentials are:\\

$\partial(z) = 
\begin{cases}
0 &  \mbox{if } z = y_{\frac{p-1}{2},*}\\
0 &  \mbox{if } z = y_{*,\frac{p-1}{2} + 1}\\
V_0 y_{*,0} + V_1 y_{p-1,*+1} &  \mbox{if } z = x_{0,*} \mbox{ or } z = x_{*,0}\\
\end{cases}$\\

Putting all differentials together, we see that, in a fixed Alexander degree the complex $GC^\infty(\lu, \widetilde{T}_{2,p},\s_0 \pm h)$ is indeed composed by two electric poles of height $h$, attached to a wire of length $p-2h$, as shown in Figure \ref{fig:complessogenerale}.

It follows that $\CFKinfty (L(p,1),\widetilde{T}_{2,p}, \s_0 \pm h)$ is isomorphic to the free $\F [U,U^{-1}]$-module generated by $C[h,p -2h]$.

Now we just need to prove that $C[h,p -2h]$ has the same homology as that of $\CFKinfty (T_{2,p-2h})$; this can be done by defining at the grid
level the quasi-isomorphisms $F_{p,h}$.

However, it is immediate to note that the change of basis shown in Figure \ref{fig:cambiobase} induces a quasi-isomorphism between $C[h,p -2h]$ and $C[h-1,p -2h]$.

\begin{figure}
\includegraphics[width=5cm]{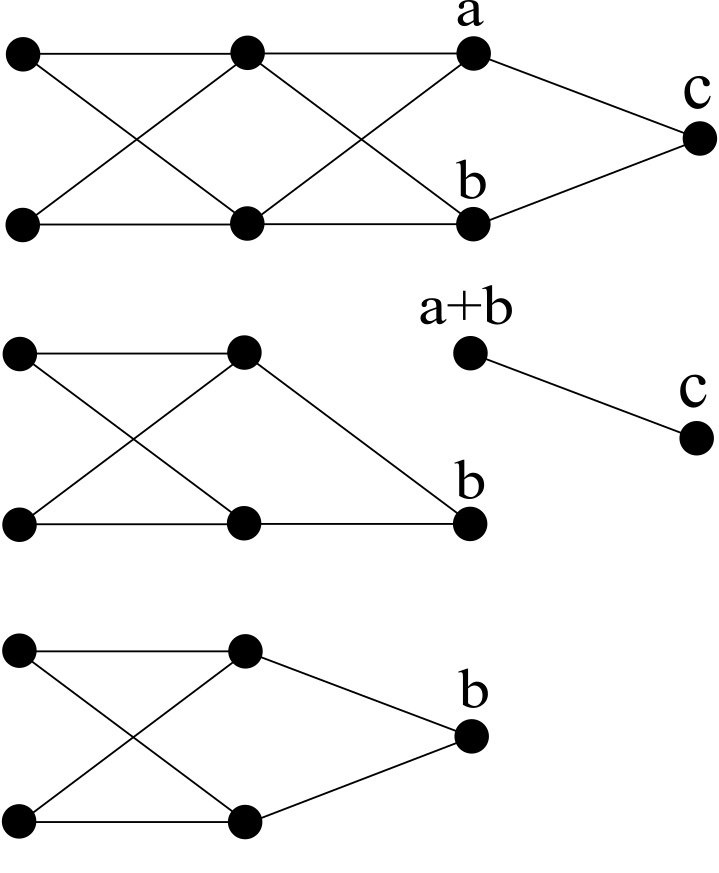}
\caption{A change of basis: $a^\prime = a+b$, $b^\prime = b$. Cancelling the acyclic pair $(c,a+b)$ lowers the height of an electric pole by $1$.}
\label{fig:cambiobase}
\end{figure}

In other words, by iterating this process $h$ times, we have proved that
\begin{equation}\label{eqn:complessiuguali}
C[e,w] \simeq C[0,w].
\end{equation}

Similar quasi-isomorphisms where used in \cite{quasiconc}, dealing with a different family of knots in $L(p,1)$.
The idea now is that by composing $h$ of these maps, we obtain a chain of quasi-isomorphisms from the complex in the $\spinc$ structure $\s_0 + h$ to the complex in the spin structure $\s_0$ in $L(p-2h,1)$.
Summing all up, we have proved the existence of the following chain of quasi-isomorphisms:
\begin{equation*}
\begin{aligned}
GC^\infty (L(p,1),\widetilde{T}_{2,p}, \s_0 \pm h) \simeq C[h,p -2h] \otimes \Z_2[V_0,V_1] \simeq C[0,p -2h] \otimes \Z_2[V_0,V_1] \simeq \\ 
\simeq GC^\infty (L(p-2h,1),\widetilde{T}_{2,p-2h}, \s_0) \simeq \CFKinfty (L(p-2h,1),\widetilde{T}_{2,p-2h}, \s_0) \simeq \CFKinfty (T_{2,p-2h}),
\end{aligned}
\end{equation*}
where 
\begin{itemize}
\item the first quasi-isomorphism is given by explicitly determining the differential of the first complex;
\item the second one is given by shortening the electric poles with the $h$ applications of the basis change in Figure~\ref{fig:cambiobase};
\item the third is just the inverse of the first one;
\item the fourth one is \cite[Theorem 1.1]{baker2008grid};
\item the last quasi-isomorphism is Grigsby's $L$.
\end{itemize}
\end{proof}
\begin{figure}[h!]
\includegraphics[width=10cm]{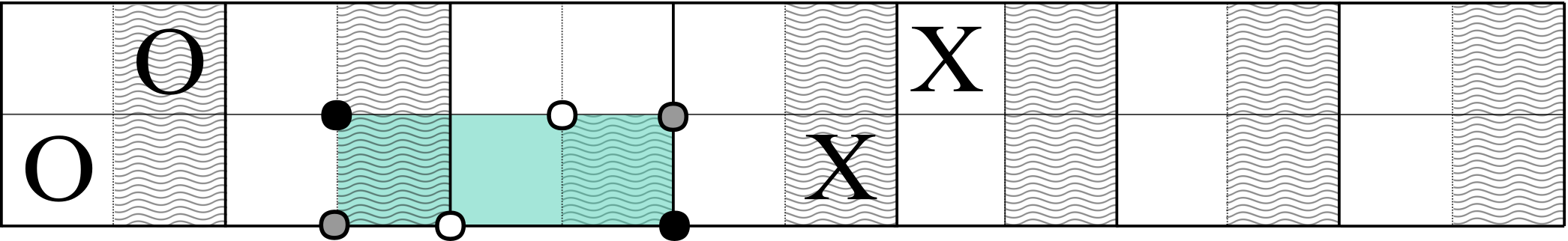}
\caption{The rectangles with wavy and filled patter start from the lower component of the gray generator and connect it to the white and black ones respectively. Only the first one has trivial intersection with the $\XX$ markings.}
\label{fig:solodue}
\end{figure}
\begin{proof}[Proof of Lemma \ref{lemma:massimodue}]
It is easy to argue that for a rectangle in $G_p$ to be empty, it must
necessarily have either height or length equal to $1$ (recall that the
top\slash bottom identification is a shift by $2$). So there can be at
most $4$ rectangles
starting from a generator $x \in
S(G_p)$.  We want to show that only two of them are disjoint from the
$\XX$ markings.

Consider the two rectangles starting from a generator, as in Figure \ref{fig:solodue}. If we assume that the \textquotedblleft horizontal\textquotedblright rectangle does not intersect the $\XX$ markings, then necessarily the other one will, and vice versa.
\end{proof}

\subsection{Consequences}
In the case of alternating knots, invariants coming from Knot Floer homology usually do not yield any interesting information.  The following result, despite being well-known \cite{litherland1979signatures}, points out an interesting way of showing that $HFK$ can be used to extract useful information from alternating knots as well.

\begin{thm}\label{thm:independencetorusalt}
The alternating torus knots $T_{2,p}$ with $p$ an odd prime are linearly independent in the concordance group.
\end{thm}
\begin{proof}
Let $K$ be of the form $K_+ \#K_-$ with $K_+=\#_i a_i T_{2,2n_i+1}$
and $K_-=- \#_j b_j T_{2, 2m_j+1}$, $a_i, b_j>0$, $n_i, m_j \geq 1$
prime and all distinct. Set $G=H_1(\Sigma(K), \Z)$,
$G_+=H_1(\Sigma(K_+), \Z)=\bigoplus_i (\Z/{(2n_i+1)}\Z)^{a_i}$, and
$G_-=H_1(\Sigma(K_-), \Z)=\bigoplus_j (\Z/{(2m_j+1)}\Z)^{b_j}$.

Assume that $K \sim \bigcirc$, then $\tau_{\s_0+ \xi}(K)=0$ for all
$\xi\in M$ for some $M\subset G$ with $|M|= \sqrt{|G|}$. Write
$\xi=(\xi_+, \xi_-)$, with $\xi_{\pm} \in G_\pm$. Because of the
additivity of the $\tau$ invariant,
one has
\begin{align*}
\tau_{\s_0+ \xi}(K)&= \tau_{\s_0+ \xi_+}\left(\#_i a_i T_{2,2n_i+1}\right) - \tau_{\s_0+ \xi_-}\left(\#_j b_j T_{2,2m_j+1} \right) \\
&= \sum_i a_i \tau_{\s_0+ \xi_i^+}\left( T_{2,2n_i+1}\right) 
- \sum_j b_j \tau_{\s_0+ \xi_j^-}\left( T_{2,2m_j+1} \right)\\
&= \sum_i a_i \tau\left( T_{2,2n_i+1-2|\xi_i^+|}\right) 
- \sum_j b_j \tau\left( T_{2,2m_j+1-2|\xi_j^-|} \right)\\
&= \sum_i a_i (n_i-|\xi_i^+|) - \sum_j b_j (m_j-|\xi_j^-|)\\
&= \left(\sum_i a_i n_i - \sum_j b_j m_j \right)
- \left( \sum_i a_i |\xi_i^+|  - \sum_j b_j |\xi_j^-| \right)
\end{align*}
Thus, if $K$ is slice, the quantity
\begin{equation}\label{eq:tauH}
H(\xi)=\sum_i a_i |\xi_i^+|  - \sum_j b_j |\xi_j^-|  \   \ \ \ \ \ \ \ \ \xi \in M \ 
\end{equation}
should be constant. In fact, since by Theorem~\ref{thm:interi} for $\xi=0$ one has $H(0)=
\tau(K)=\sum_i a_i n_i - \sum_j b_j m_j =0$, for $\xi\in M$ we have
$H(\xi)\equiv 0$. On the other hand, because of the assumption on the
coefficients of the torus knots being primes, one can find non-zero
$\xi\in M$ of the form $\xi=(\xi_+, 0)$. However, for such a $\xi$ one
has $H(\xi)>0$, and we reach a contradiction.
\end{proof}

\begin{rem}
The argument given above (and indeed, the statement of
Theorem~\ref{thm:independencetorusalt}) is not optimal. The same argument works if in a family $T_{2,2n_i+1}$ all $2n_i+1$ are powers of different
primes. The fact that all
alternating torus knots are linearly independent, however, requires
a more subtle group theoretic result, which we do not pursue here.
\end{rem}
%
%
%

\begin{rem}
Using a result by Raoux~\cite{raoux2016tau}, we can determine the
slice genus of the lifts $\widetilde{T}_{2,p}$. Here by slice genus
$g_*(K)$ of a knot $K \subset \lu$, we mean the minimal genus of a
smooth and properly embedded surface $\Sigma$ in $(W_p, \lu)$, and
such that $\partial \Sigma = K$, where $W_p$ is the Euler number $p$ disk
bundle over $S^2$.  

If $K \subset S^3$ bounds a smooth slice surface
$S$ of genus $g$, then its lift $\widetilde{K} \subset \Sigma(K)$ will
bound a genus $g$ surface in the double branched cover of $D^4$ over
$S$, obtained by lifting $S$. Hence $g_*(\widetilde{T}_{2,p}) \le
n=\frac{1}{2}(p-1)$. On the other hand, since $\tau^\s(K) \le g_*(K)$ for every
$\spinc$ structure that extends to $W_p$, as does $\s_0$, using
\cite[Corollary~5.4]{raoux2016tau} we have $g_* (\widetilde{T}_{2,p})
= n$. 
\end{rem}

\begin{rem}
By \cite[Proposition~25]{quasiconc}, the knots
$\widetilde{T}_{2,p}$ are not concordant to local knots. Also, by
\cite[Theorem~4]{quasiconc} the genus of a $PL$ surface cobounding
$\widetilde{T}_{2,p}$ and $\bigcirc$ in $L(p,1) \times I$ is at least
$n$. Here, by $PL$ surface, we mean a surface which is smooth
everywhere, except for a finite number of singular points, which are
cones over knots in $S^3$.
\end{rem}

\section{Further examples with (1,1)-knots} \label{sec:TorusAgain}

In this section we perform some further computations for invariants
in double branched covers for knots given by 
 genus one doubly pointed Heegaard diagrams --- such knots are known
as $(1,1)$-knots \cite{Ras1}. Based on these computations we give an
alternative proof of Theorem~\ref{thm:independencetorusalt}, and prove
an independence result about twist knots.

\subsection{$(1,1)$-diagrams of alternating torus knots}
Note first that the alternating torus knot $T_{2,2n+1}$ can be given
by the toroidal doubly pointed Heegaard diagram in the top-left of
Figure~\ref{fig:t22np1Comp}. The generators of the chain complex of
$\CFKinf (K)$ for a $(1,1)$-knot are easy to determine: these are the
intersection points of the unique $\alpha$- and the unique
$\beta$-curve. The boundary map is defined by counting holomorphic
maps from the unit disk ${\mathbb {D}}$ to the Heegaard torus $T^2$,
with the usual (Floer theoretic) boundary conditions. The map
${\mathbb {D}}\to T^2$ is not necessarily injective, but once we pass
to the universal cover ${\mathbb {C}}$ of $T^2$ (and lift one of the
intersection points to one of its preimages), the boundary map can be
determined by identifying embedded bigons in the universal cover.  As
we will see, in some cases it is sufficient to consider one or two
fundamental domains, making the computation much simpler.

From Figure~\ref{fig:t22np1Comp} one can easily find the chain complex
$\CFKinfty(T_{2,2n+1})$, shown in the lower part of
Figure~\ref{fig:t22np1Comp}.
\begin{figure}
\includegraphics[width=12cm]{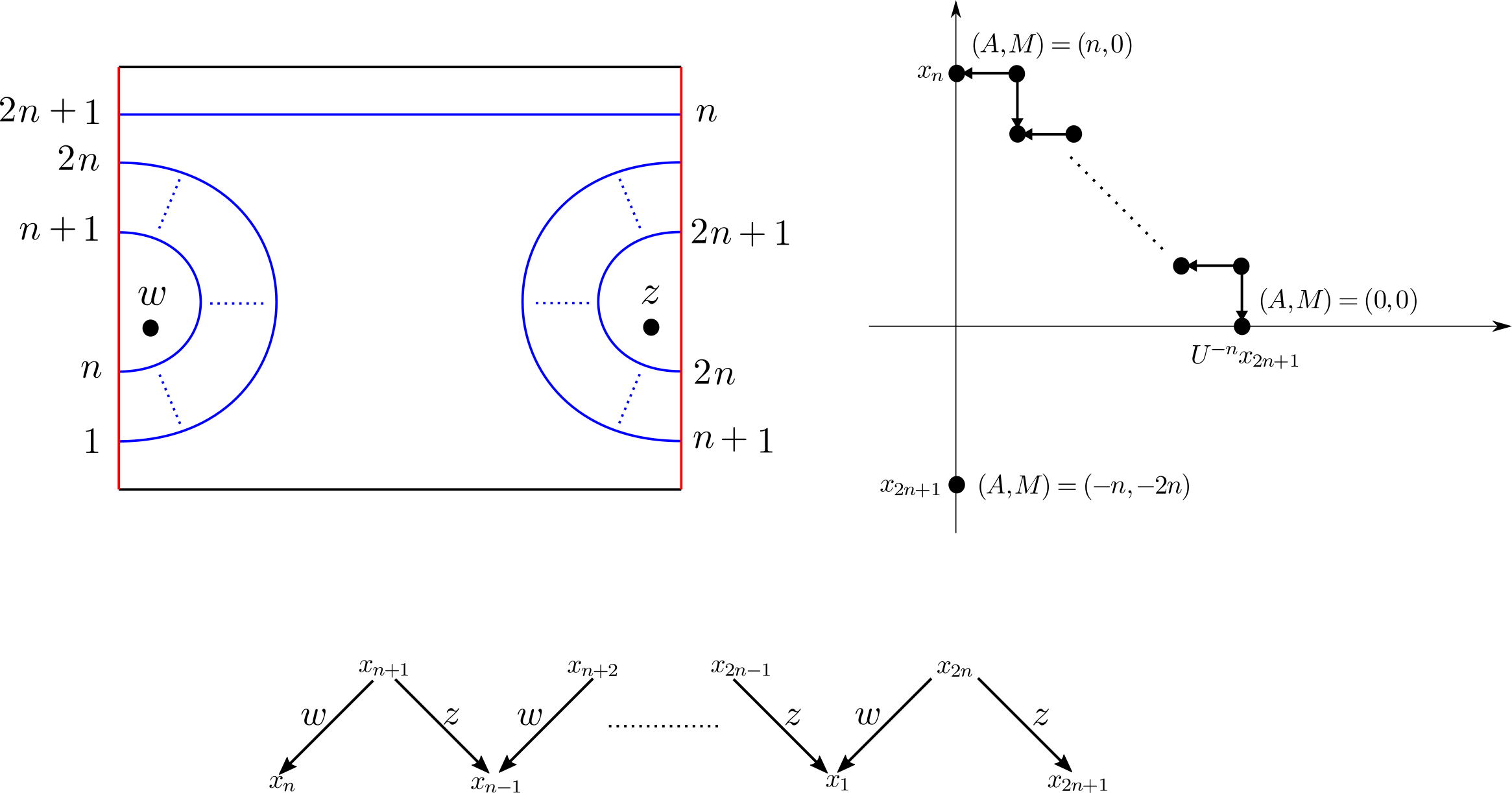}
\caption{On the left, the $(1,1)$-diagram of the torus knot
  $T_{2,2n+1}$ is given; we get the toroidal diagram after identifying
the top and the bottom edges, as well as the right and left edges, in the latter identfication we apply a shearing, so that points with equal label match.
On the right we show the chain complex (all
dots and arrows should be shifted by $(k,k)$ ($k\in \Z$) for the complete
picture), and in the lower part the boundary map is shown. As usual,
the intersection point $i$ corresponds to the generator $x_i$.
(The letter $w$ or $z$ over an arrow indicates that the corresponding
bigon contains the basepoint shown.)  
Note that in the notation of the picture we have
  $HFK^-(T_{2,2n+1})\simeq \F[U] \oplus \F^n$, where the first summand is
  generated by $x_{2n+1}$, and the latter by $x_1, \ldots, x_{n}$.}
\label{fig:t22np1Comp}
\end{figure}
The generator corresponding to the intersection point $i\in \{1,
\ldots , 2n+1\}$ is denoted by $x_i$ for $i=1, \ldots ,{2n+1}$, and
the boundary maps are given by the bigons visible on the picture.
(There are no more non-trivial components of the boundary map, since the
Maslov index one domains connecting any further pairs contain regions
of negative multiplicity.)

The double branched cover $\Sigma (T_{2,2n+1})$ is a 3-manifold
diffeomorphic to the lens space $L(2n+1, 1)$. The lift of the branch
curve $T_{2,2n+1}\subset S^3$ provides a null-homologous knot
${\widetilde {T}}_{2,2n+1}\subset \Sigma (T_{2,2n+1})$.  The pullback
of $\alpha , \beta \subset T^2$ to the double branched cover will be
denoted by $\alpha _1, \alpha _2$ and $\beta _1, \beta _2$.  In this
way we get a genus-2 Heegaard diagram ${\mathcal {H}}=(\Sigma _2, \{
\alpha _1, \alpha _2\}, \{ \beta _1, \beta _2\})$ for $\Sigma
(T_{2,2n+1})$, as $\Sigma _2$ is the double branched cover of the
torus branched in $z$ and $w$.  The intersection points ${\mathbb
  {T}}_{\alpha }\cap {\mathbb {T}}_{\beta}$ can be easily described as
follows.  Let $a_i, b_i$ denote the two points over $x_i$, and assume
that $a_i\in \alpha _1$ (and so $b_i\in \alpha _2$).  We also assume
that $a_1\in \alpha _1\cap \beta _1$ (and hence $b_1\in \alpha _2\cap
\beta _2$). There are two types of points $x_i$: we say that $x_i$ is
\emph{homogeneous} if $a_i\in \beta _1$ (so by our choice $x_1$ is
homogeneous) and \emph{inhomogeneous} if $a_i\in \beta _2$.

Since the bigons in the
fundamental domain appearing in the boundary maps all lift to
rectangles (all these bigons contain a unique basepoint), they
connect homogeneous points with inhomogeneous ones, hence we conclude
that $x_1, \ldots, x_n$ and $x_{2n+1}$ are homogeneous and $x_{n+1},
\ldots, x_{2n}$ are inhomogeneous. 

This shows that the pull-back
diagram in the double branched cover has $(n+1)^2+n^2$ generators in
the Heegaard Floer chain complex: these are $\{ (a_i, b_j)\}$ where
either both $x_i,x_j$ are homogeneous or both inhomogeneous. It is easy
to see that $(a_i,b_i)$ and $(a_j, b_j)$ are in the same spin$^c$
structure (since the bigons connecting the various $x_i's$ lift to
rectangles connecting the various $(a_i,b_i)$'s). Indeed

\begin{lem}\label{lem:SpinCForTorus}
The generators $(a_i,b_i)$ represent the spin structure $\s _0$ of $\Sigma
(T_{2,2n+1})$. The pairs $(a_i, b_j)$ and $(a_k, b_{\ell})$
represent the same spin$^c$ structure if and only $i-j=k-\ell$.
\end{lem}
\begin{proof}
The elements $(a_i, b_i)$ all represent the same spin$^c$ structure
since they can be connected by domains (the pull-backs of the bigons
from downstairs), and since this spin$^c$ structure is conjugation
invariant, it must be the unique spin structure $\s _0$.  Now we can 
compare the spin$^c$ structure of 
$(a_i, b_j)$ to that of $(a_i, b_i)$: the difference will be represented
by the lift of the 1-homology in the torus we get by connecting $x_i$ to $x_j$ on the $\alpha$-curve and $x_j$ to $x_i$ on the $\beta $-curve.
This implies that the spin$^c$ structure
of $(a_i, b_j)$ is given by twisting $\s _0$ with $(j-i)$-times a
generator of $H_1$ of the double branched cover, which concludes the
proof.
\end{proof}

Recall that the pairs come with Alexander and Maslov gradings, and 
indeed the Alexander gradings are relatively easy to compute in 
terms of the Alexander gradings of the $x_i$'s.

\begin{lem}[Levine, \cite{levine2008computing}]
\label{lem:LevinsResult}
The Alexander grading $A(a_i,b_j)$ 
of $(a_i,b_j)$ is equal to $\frac{1}{2}(A(x_i)+A(x_j))$.
In particular, $A(a_n, b_n)=n$, $A(a_{2n+1}, b_{2n+1})=-n$ and 
for all further generator $\vert A(a_i,b_j)\vert <n$. \qed
\end{lem}

The knot Floer complex for 
$(\Sigma (T_{2,2n+1}), {\widetilde{T}}_{2,2n+1}, \s _0)$ 
for the spin structure $\s _0$ is actually
isomorphic to the knot Floer complex of $T_{2,2n+1}\subset S^3$:
all non-trivial components of $\partial $ in $\CFKinf (T_{2,2n+1})$ are
defined by bigons in the fundamental domain of the $(1,1)$-diagram we
are working with, and these bigons lift to rectangles in the double
branched cover (since each bigon contains a unique $z$ or $w$), hence
we get the same maps upstairs in $\CFKinf (\Sigma (T_{2,2n+1}),
{\widetilde {T}}_{2,2n+1},\s _0)$. 
The same argument 
as for $T_{2,2n+1}\subset S^3$ then applies and
shows that there are no more non-zero components of the boundary map
upstairs: any domain with Maslov index 1 connecting further pairs of
generators has some negative multiplicity.
In this way we get the $\Upsilon$-invariant of 
$(\Sigma (T_{2,2n+1}), {\widetilde{T}}_{2,2n+1}, \s _0)$.

\begin{proof}[Proof of Theorem~\ref{thm:independencetorusalt}]
The proof is essentially the same as the proof given at the end of
Section~\ref{sec:AltTorus}: for a relation we consider the double
branched cover along the slice knot, which admits a filling by a
rational homology disk, and for all spin$^c$ structures extending from
the double branched cover to this 4-manifold, the same linear
combination of $\Upsilon$-functions as in the relation must vanish.
The spin structure extends, hence we get one relation, and then we
find an element of the form $(h_+, 0)$ in the metabolizer $G_+\oplus
G_-=G$, where $h_+$ has some non-zero component. By
Lemma~\ref{lem:LevinsResult} at that component we replace the
corresponding value of $\Upsilon (1)$ with some smaller value, which
will result in violating Equation~\eqref{eq:tauH}, concluding
the proof.
\end{proof}

\subsection{Twist knots}
The $(1,1)$-diagram of Figure~\ref{fig:twist} provides a doubly
pointed Heegaard diagram for the twist knot $\twistknot _n$. The parameter
$n>0$ is chosen so, that $\twistknot _1$ is the (right-handed) trefoil knot and
$\twistknot _2$ is the Figure-8 knot.

\begin{figure}[h!]
\includegraphics[width=8cm]{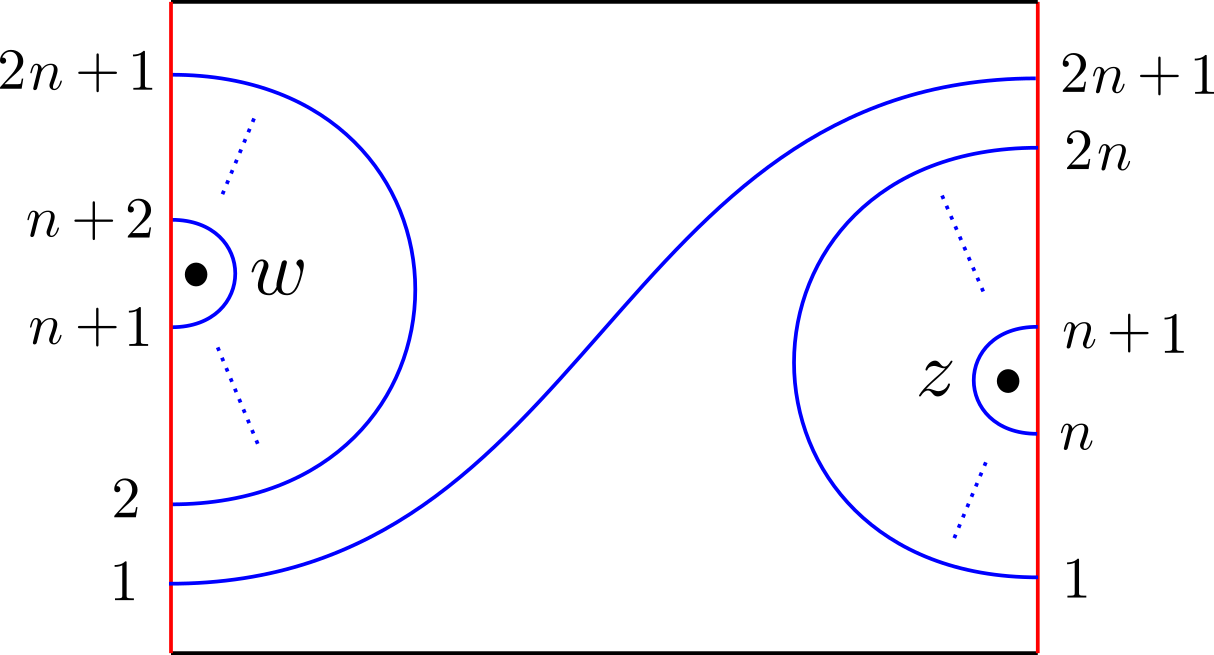}
\caption{$(1,1)$-diagram of the twist knot $\twistknot _n$}
\label{fig:twist}
\end{figure}
In particular, the determinant $\det (\twistknot _n)$ is equal to $2n+1$,
and for $n$ odd we have
\[
\Delta _{\twistknot _n}(t)=\frac{n+1}{2}t-n+\frac{n+1}{2}t^{-1}, 
\]
while for $n$ even
\[
\Delta _{\twistknot _n}(t)=-\frac{n}{2}t+(n+1)-\frac{n}{2}t^{-1}. 
\]
Furthermore, the signature $\sigma (\twistknot _n)=-1$ if $n$ is odd, 
and $\sigma (\twistknot _n)=0$ if $n$ is even.

From the diagram we can easily determine the chain complex $\CFKinf$.
This can be done by analyzing the bigons in the universal cover
of the $(1,1)$-diagram -- in this particular case, in fact, two fundamental
domains will suffice to contain all relevant bigons, see
Figure~\ref{fig:twotwist}.
\begin{figure}[h!]
\includegraphics[width=12cm]{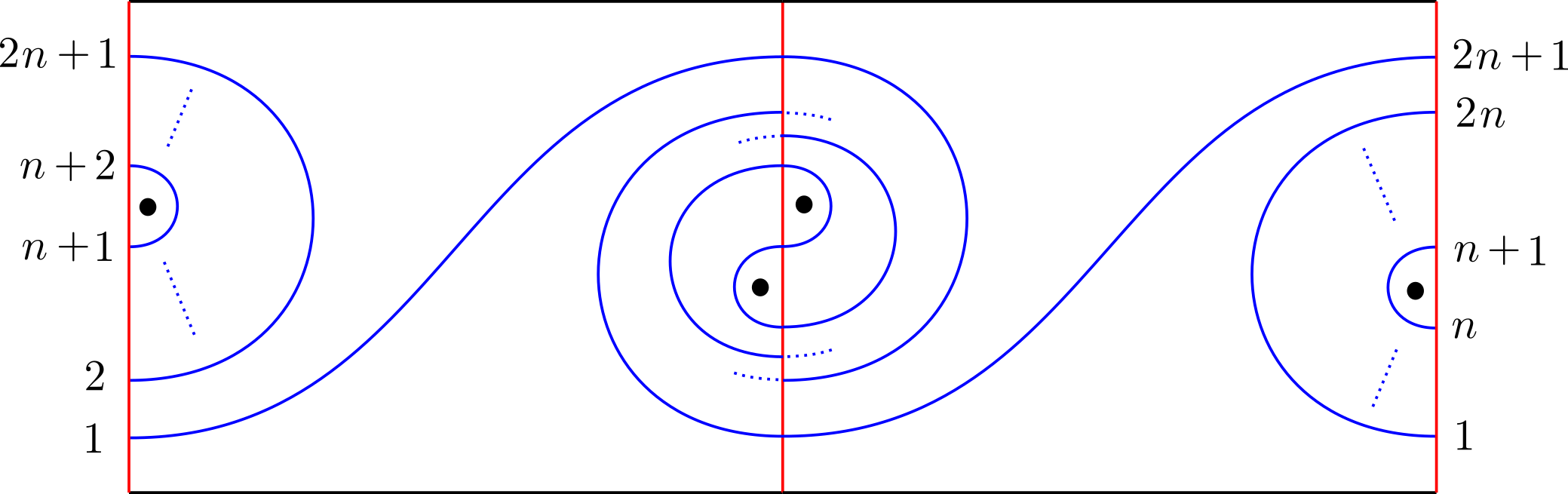}
\caption{Two fundamental domains in the universal cover of the $(1,1)$-diagram of $\twistknot _n$.}
\label{fig:twotwist}
\end{figure}

For the schematic picture of the chain complex see
Figure~\ref{fig:twistchaincomplex}.
\begin{figure}
\includegraphics[width=3cm]{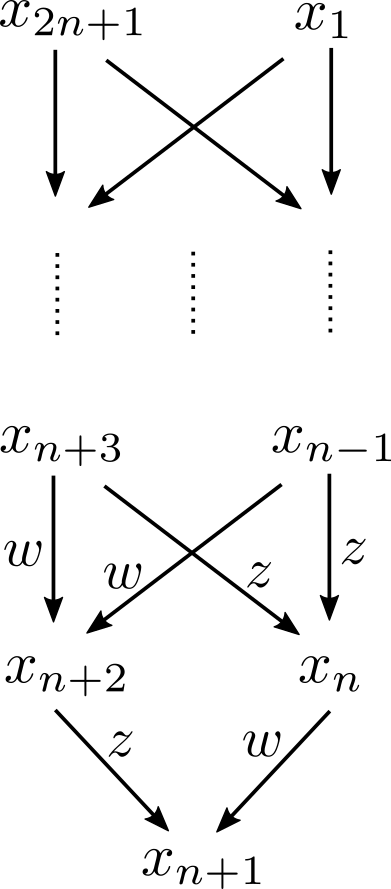}
\caption{The chain complex of the twist knot $\twistknot _n$. We use the
same conventions as in Figure~\ref{fig:t22np1Comp}. Recall that the decorations
$z$ or $w$ over the arrows indicate which basepoint is crossed by the bigon
corresponding to the particular component of the boundary.}
\label{fig:twistchaincomplex}
\end{figure}
Denoting the intersection point corresponding to $i$ by $x_i$ again, we get 
the following:

\begin{lem}
The Alexander grading $A(x_i)$ of $x_i$ is
\begin{itemize}
\item $0$ if $i\equiv n+1 \pmod{2}$,
\item $1$ if $i\equiv n \pmod{2}$ and $i>n+1$, and 
\item $-1$ if $i\equiv n \pmod{2}$ and $i<n+1$.
\end{itemize}
\end{lem}
\begin{proof}
Recall that $A(\x)-A(\y)=n_z(\phi )-n_w(\phi )$ for a domain $\phi $
connecting $\x$ and $\y$. If there is an arrow from $x_i$ to $x_j$,
then $A(x_i)-A(x_j)=1$ if the arrow is decorated by $z$ and is equal
to $-1$ if the decoration of the arrow is $w$.  Since on the vertical
paths of Figure~\ref{fig:twistchaincomplex} the decorations alternate,
and (by symmetry) $A(x_{n+1})=0$, the claim follows at once.
\end{proof}

As before, in the double branched cover the $\alpha$- and the
$\beta$-curves lift to $\alpha _1,\alpha _2$ and $\beta _1, \beta
_2$. Similarly, each intersection point $x_i$ gives rise to two
intersections $a_i$ and $b_i$. We will follow the convention that
$a_i\in \alpha _1$, $b_i\in\alpha _2$, and moreover $a_1\in \beta _1$.
Then it is not hard to see that $a_i\in \alpha _1\cap \beta _1$ if and
only if $i$ is odd (and then $b_i\in \alpha _2\cap \beta _2$) and for
even $i$ we have $a_i\in \alpha _1\cap \beta _2$ and $b_i\in \alpha
_2\cap \beta _1$. Consequently the generators of the Heegaard Floer
chain complex of the double branched cover are of the form $(a_i,b_j)$
with the constraint that $i\equiv j \pmod{2}$.

According to the earlier cited result of Levine, we have that 
$A(a_i,b_j)=\frac{A(x_i)+A(x_j)}{2}$.

\begin{lem}
The spin$^c$ structure $\s (a_i,b_j)$ can be determined as follows:
\begin{itemize}
\item $(a_i,b_j)$ represents the spin structure $\s _0$ on $\Sigma
  (\twistknot _n)$ if and only if $i=j$, and
\item $(a_i,b_j)$ and $(a_k, b_{\ell})$ are in the same spin$^c$ structure
if and only if $i-j=k-\ell$.
\end{itemize}
\end{lem}
\begin{proof}
The proof is the direct adaptation of the proof of
Lemma~\ref{lem:SpinCForTorus}.
\end{proof}
A simple corollary of the above provides:
\begin{cor}
Suppose that $\vert i -j\vert >n+1$. Then $A(a_i,b_j)=0$. \qed
\end{cor}

According to the last corollary, for $n\geq 1$ there is always a 
spin$^c$ structure where all generators have $A=0$, hence the
Upsilon-function is constantly zero.

On the other hand, as in the case of torus knots, we immediately see
that in the spin structure $\s _0$ the chain complex $\CFKinf
(\Sigma(\twistknot _n), {\widetilde {\twistknot }}_n, \s _0)$ is isomorphic to $\CFKinf
(\twistknot _n)$. Indeed, the previous lemma shows a bijection between the
generators, and all bigons of the fundamental domain lift to an
embedded rectangle, providing non-trivial boundary maps.  (The further
non-zero components then are forced by the fact that $\partial ^2=0$.)

Independence of the family $\{\twistknot _n\}_{n>0, n \neq 4}$ was
first established in~\cite{CG}, and can be deduced from Lisca's
work~\cite{lisca2007sums}.
The above data suffice to show a weaker result about 
linear independence in the concordance group:
\begin{prop}
The family $\{ \twistknot _p\mid p$ odd and $2p+1$ is prime$\}$ forms a set of linearly 
independent elements in the smooth concordance group ${\mathcal {C}}$.
\end{prop}

\begin{proof}
As in the case of torus knots, suppose that 
there exists a linear dependence $\sum _i k_i \twistknot _{p_i}\sim 0$, and rewrite
it as $\sum m_i \twistknot _{p_i} -\sum n_j \twistknot _{q_j}\sim 0$
with $m_i, n_j>0$. 

The double branched cover $Y$ (which is the 
appropriate connected sum of the double branched covers of 
the individual knots $\twistknot _{p}$), bounds a rational homology ball,
hence in $H_1(Y; \Z)$ there is a metabolizer $M$ as before.
Once again, we write $M$ as $M_+\oplus M_-$. 

Notice that since the
spin structure extends, we have that 
$\sum m_i-\sum n_j=0$, or equivalently
\begin{equation}\label{eq:eqqq}
\sum m_i=\sum n_j.
\end{equation}
(This would also follow from the signature values of the knots 
since we assumed that all $p$ are odd.)

Now suppose that $(h_+, h_-)$ is a non-zero element in the metabolizer
$M$. By the same argument as in the proof of
Theorem~\ref{thm:independencetorusalt}, an appropriate multiple of
this element looks like $(0,h_-')$ (or $(h_+',0)$, but the two cases
are completely symmetric). By eventually taking a further multiple, we
can assume that at least one component of $h_-'$ is in a spin$^c$
structure where all generators have Alexander grading equal to zero,
hence the $\Upsilon$-function in that spin$^c$ structure is
identically zero.

This provides the desired contradiction, since  in Equation~\eqref{eq:eqqq}
we do not change the left hand side, but delete some (strictly positive)
terms from the right hand side, hence the resulting expression does not hold anymore, contradicting the fact that the spin$^c$ structure
extends to the rational homology 4-ball.
\end{proof}

\begin{rem}
We expect that for $n$ even the knot ${\widetilde {\twistknot }}_n\subset
\Sigma (\twistknot _n)$ has constant zero $\Upsilon$-function in every spin$^c$
structure.
\end{rem}

\end{document}